\documentclass[11pt, leqno]{article}

\usepackage{amsmath,amsthm,amsfonts,amssymb,tabularx}
\usepackage[latin1]{inputenc}
\usepackage{graphicx}
\usepackage{color}
\usepackage{nameref}
\usepackage{fancyhdr}
\usepackage{subfigure}
\usepackage{amsmath}
\usepackage{mathrsfs}
\usepackage{amssymb}
\usepackage{setspace}

\usepackage{caption}
\usepackage{multirow}
\usepackage{setspace}

\usepackage[inner=3cm,outer=4.5cm,bottom=4cm,top=4cm]{geometry}

\usepackage{hyperref}
\usepackage{pdfpages}

\hypersetup{urlcolor=blue, colorlinks=true}

\newtheorem{theorem}{Theorem}[section]
\newtheorem{proposition}[theorem]{Proposition}

\newtheorem{lemma}[theorem]{Lemma}

\newtheorem{remark}{Remark}
\newcommand\R{\mathbb{R}}

\newcommand{\mI}{\mathcal{I}}

\newcommand\pp{\tilde{p}}
\newcommand\llll{\tilde{l}}

\newcommand\N{\mathbb{N}}

\numberwithin{equation}{section}

\newcommand{\B}{\mathcal{B}}
\newcommand{\HH}{\mathbb{H}_\B}
\newcommand{\LBs}{\mathcal{L}_{g}^s}

\newcommand{\DB}{(-\Delta_\B)^s}
\newcommand{\DBm}{(-\Delta_\B)^{-s}}

\newcommand{\RN}{\mathbb{R}^n}

\newcommand{\Addresses}{{% additional braces for segregating \footnotesize
  \bigskip
  \footnotesize

  Nicole Cusimano, \textsc{Basque Center for Applied Mathematics, Bilbao, Spain.}\par\nopagebreak
  \textit{E-mail address:} \texttt{ncusimano@bcamath.org}

  \medskip

  F\'elix del Teso, \textsc{Basque Center for Applied Mathematics, Bilbao, Spain.}\par\nopagebreak
  \textit{E-mail address:} \texttt{fdelteso@bcamath.org}

  \medskip
  Luca Gerardo-Giorda, 
\textsc{Basque Center for Applied Mathematics, Bilbao, Spain.}\par\nopagebreak
  \textit{E-mail address:} \texttt{lgerardo@bcamath.org}

}}

\begin{document}
%\begin{frontmatter}
\title{\bf{Numerical approximations for fractional \\elliptic equations via the method of semigroups}}

\author{\bf Nicole Cusimano, F\'elix del Teso, Luca Gerardo-Giorda}

\maketitle

\begin{abstract}
\noindent We provide a novel approach to the numerical solution of the family of nonlocal elliptic equations $(-\Delta)^su=f$ in $\Omega$, subject to some homogeneous boundary conditions $\B(u)=0$ on $\partial \Omega$, where $s\in(0,1)$, $\Omega\subset \mathbb{R}^n$ is a bounded domain, and $(-\Delta)^s$ is the spectral fractional Laplacian associated to $\B$ on $\partial \Omega$. We use the solution representation $(-\Delta)^{-s}f$ together with its singular integral expression given by the method of semigroups. By combining finite element discretizations for the heat semigroup with monotone quadratures for the singular integral we obtain accurate numerical solutions. Roughly speaking, given a datum $f$ in a suitable fractional Sobolev space of order $r\geq 0$ and the discretization parameter $h>0$, our numerical scheme converges as $O(h^{r+2s})$, providing super quadratic convergence rates up to $O(h^4)$ for sufficiently regular data, or simply $O(h^{2s})$ for merely $f\in L^2(\Omega)$. We also extend the proposed framework to the case of nonhomogeneous boundary conditions and support our results with some illustrative numerical tests.
%$r\in[0,4-2s]$ 

\end{abstract}
{
\begin{spacing}{0.001}
\setcounter{tocdepth}{1}
\tableofcontents
\end{spacing}
}
\vspace{10mm}


% REQUIRED

%\end{frontmatter}

\section{Introduction}
\label{sec:intro}
The aim of this work is to provide numerical approximations to the solution of the following fractional elliptic problem:
\begin{align}
\label{EP}
(-\Delta)^su(x)&=f(x)  \quad \quad \text{in}\qquad \Omega \\
\label{BC} 
\B (u) &=0  \ \qquad \quad \text{on}\qquad \partial \Omega ,
\end{align}
where $s\in(0,1)$, $\Omega\subset \mathbb{R}^n$ is a bounded domain, and $(-\Delta)^s$ is the spectral fractional Laplacian associated to some homogeneous boundary conditions $\B$ on $\partial \Omega$. We stress that the fractional operator is boundary condition dependent and hence from now on in this manuscript we always denote it by $(-\Delta_{\B})^s$. 

We specifically consider Dirichlet, Neumann, and Robin boundary conditions, corresponding respectively to the boundary operators
\begin{equation}\label{as:BC}\tag{BC}
\mathcal{B}(u)=u, \quad \mathcal{B}(u)=\frac{\partial u}{\partial \nu}, \quad \textup{and} \quad \mathcal{B}(u)=\kappa~u+\frac{\partial u}{\partial \nu}, \quad \kappa  >0 \quad  \textup{on}  \quad  \partial \Omega.
\end{equation}
 However, the results and techniques of this paper could be extended to fractional powers of more general operators with appropriate boundary conditions.

Let $\{\varphi_m, \lambda_m\}_{m=1}^\infty$ denote the eigenpairs of $-\Delta_\B$, that is, the nontrivial solutions of the eigenvalue problem
\begin{align*}
-\Delta \varphi_m&=\lambda_m \varphi_m \qquad  \textup{in} \quad \Omega\\
\B(\varphi_m)&=0   \quad \qquad \ \ \ \textup{on} \quad \partial \Omega.
\end{align*}
Throughout the manuscript we will assume the $\varphi_m$ to be normalized, i.e., $\int_{\Omega}\varphi_m^2~dx=1$.
Then, for all functions $u\in L^2(\Omega)$ such that $\sum_m \lambda_m^{2s} |\hat{u}_m |^2 <\infty$ with $\hat{u}_m:=\int_{\Omega} u(x)\varphi_m(x) dx$, the spectral fractional Laplacian $\DB$ is defined for $s\in (0,1)$ by 
\begin{equation}\label{eq:FLdef1}
\DB u(x)=\sum_{m=1}^\infty \lambda_m^s \hat{u}_m \varphi_m(x).
\end{equation}
Equivalently, $\DB$ can be defined through an integral formulation as follows:
\begin{equation}\label{eq:FLdef2}
\DB u(x)=\frac{1}{\Gamma(-s)}\int_0^\infty\left( e^{t\Delta_\B}u(x)-u(x)\right)\frac{d t}{t^{1+s}},
\end{equation}
where $w(x,t):=e^{t\Delta_\B}f(x)$ denotes the heat semigroup associated to the operator $\DB$, that is, the solution to the following heat equation:
\begin{align}\label{eq:HeatEq}
w_t(x,t)-\Delta w(x,t)&=0 \quad \quad \quad \text{in}\qquad \Omega\times(0,\infty), \nonumber\\
\B (w)(x,t) &=0   \qquad \quad \text{on}\qquad \partial \Omega\times(0,\infty),\\
w(x,0)&=f(x) \ \ \quad \text{in}\qquad \Omega.\nonumber
\end{align}

For suitable functions $f$, the solution of \eqref{EP}-\eqref{BC} can be expressed in terms of negative fractional powers of $-\Delta_\B$ as
\begin{equation}\label{eq:IFLspec}
u(x)=\DBm f(x):= \sum_{m=1}^\infty \lambda_m^{-s} \hat{f}_m \varphi_m(x),
\end{equation}
which in turn have an equivalent integral formulation involving the heat semigroup as follows:
\begin{equation}\label{eq:IFL}\tag{S}
(-\Delta_\B)^{-s} f(x)=\frac{1}{\Gamma(s)}\int_0^\infty e^{t\Delta_\B}f(x)\frac{d t}{t^{1-s}}.
\end{equation}
\begin{remark} 
	The equivalence of \eqref{eq:IFL} and \eqref{eq:IFLspec} is proved using the expression of the heat semigroup in terms of the eigenpairs 
\begin{equation}
\label{eq:Wseries}
e^{t\Delta_\B}f(x)=\sum_{m=1}^\infty e^{-\lambda_m t}\hat{f}_m \varphi_m(x),
\end{equation}
and the following identity valid for all $\lambda,a>0$:
\begin{equation}\label{eq:numPower}
\frac{1}{\Gamma(a)}\int_0^\infty e^{-t\lambda}\frac{d t}{t^{1-a}}= \lambda^{-a}.
\end{equation}
\end{remark}

\noindent {\bf{Notation and functional setting.}} When working with $\DB$, it is natural to consider functional spaces defined for $r\geq 0$ as
\begin{equation}\label{eq:defHH}
\HH^r(\Omega):=\left\{ \psi \in L^2(\Omega)\colon \|\psi\|_{\HH^r(\Omega)}:=\left(\sum_{m=1}^\infty\lambda_m^r|\hat{\psi}_m|^2\right)^{1/2}<\infty \right\},
\end{equation}
where  $\lambda_m$ are the eigenvalues of $-\Delta_\B$ and $\hat{\psi}_m$ are $L^2$-products of $\psi$ and the eigenfunctions $\varphi_m$. From now on in this manuscript we simply write $\|\cdot\|_r$ when referring to the norm of $\HH^r(\Omega)$ and notice in particular that $\|\cdot\|_0=\|\cdot\|_{L^2(\Omega)}$ since $\HH^0(\Omega)=L^2(\Omega)$ by definition. 

Let $H^q(\Omega)$ denote the standard Sobolev space of order $q$. On smooth domains one can show that for any non-negative integer $q$, the spaces $\HH^q(\Omega)$ can be characterized as follows (see \cite{Vid97} for the case of homogeneous Dirichlet boundary conditions):
\[
\HH^q(\Omega)=\{\psi \in H^q(\Omega) \colon \B(\Delta^j \psi)=0 ~\text{on} ~\partial \Omega ~~\text{for all non-negative integers } j<q/2 \}.
\]
On general bounded domains with Lipschitz boundary, the above characterization not necessarily holds but some useful connections between $\HH^r(\Omega)$ and $H^r(\Omega)$ can still be made (see relevant comments in Chapter 19 of~\cite{Vid97}, Section 4 in~\cite{Bon2015}, or Section 2 in~\cite{NoOtSa15}).

In light of~\eqref{eq:FLdef1}, $\DB u$ is defined for any $u\in \HH^{2s}(\Omega)$. Moreover, given $f\in \HH^{r}(\Omega)$ one expects that $\DBm f\in \HH^{r+2s}(\Omega)$ (e.g., see~\cite{CaSt16}).  The improved regularity of the solution with respect to the right-hand side of the equation will be somehow reflected in the order of convergence of our numerical approximation. 

Throughout the manuscript we denote by $C$ several different constants for notational simplicity, and use the notation $a\sim b$ when $a=Cb$ for some constant $C$ independent of $b$.

\begin{remark} If $r_1<r_2$, from~\eqref{eq:defHH} we immediately see that  $\HH^{r_2}(\Omega)\subset \HH^{r_1}(\Omega)$ and $\|\psi\|_{r_1}\leq  C\|\psi\|_{r_2}$.
\end{remark}

\noindent {\bf{Brief description of techniques and results.}} In this work we use the characterization of $\DBm f$ given by \eqref{eq:IFL} to provide a numerical approximation for the solution of the fractional elliptic problem \eqref{EP}-\eqref{BC}. More precisely, we approximate $\DBm f$ by introducing suitable monotone quadratures for the integral with respect to the singular measure $\frac{dt}{t^{1-s}}$ in \eqref{eq:IFL} and by using the method of Finite Elements (FE) to approximate $e^{t\Delta_\B}f$. 

To derive our approximation to $\DBm f$, we start by proposing and studying monotone quadrature rules for integrals of the form $\int_0^T \rho(t)\frac{dt}{t^{1-s}}$. We do this by considering a time discretization parameter $\varDelta t>0$ and suitable interpolating functions $\mI[\rho](t)$ chosen according to the regularity of $\rho$. We then apply these quadratures to the case in which $\rho(t)=e^{t\Delta_\B}f$ and obtain an approximation error $O(\varDelta t^{\frac{r+2s}{2}})$ for $f\in \HH^{r}(\Omega)$ with $r\in[0,4-2s]$. 

The tail of the integral representation of $\DBm f$, namely $\int_T^\infty e^{t\Delta_\B}f(x)\frac{d t}{t^{1-s}}$, is simply truncated after a suitable choice of $T$. More precisely, the exponential decay in time of $e^{t\Delta_\B}f$ allows us to select $T\sim \log(\varDelta t^{-1})$ and to obtain an error comparable to the one given by the quadrature approximation of the rest of the integral. 

The analytic expression of $e^{t\Delta_\B}f$ involves the eigenpairs of $-\Delta_\B$ which on general domains are not known explicitly. We consider the classical theory of FE (e.g., see~\cite{Vid97}) to produce suitable approximations of the heat semigroup. By exploiting well-established error estimates with optimal error decay depending on the regularity of the initial condition $f$, we finally prove convergence of our approach.
In particular, given a spatial discretization parameter $h>0$, a right-hand side $f\in \HH^{r}(\Omega)$ and a FE approximation of order $k+1$, our numerical method produces errors $O(h^{\min\{k+1,r+2s\}})$.

In this work, we also consider a modification of \eqref{EP}-\eqref{BC} in which the nonlocal operator is coupled to non-homogeneous boundary conditions and show how to adapt our techniques to this framework. We conclude by providing some numerical experiments corroborating our main results in different cases of interest.

\noindent{\bf{Relevant works.}} Fractional order differential operators and the method of semigroups are classical analysis subjects that can be found in many seminal books, such as \cite{Kat80,Yos80,Paz82}. Nevertheless, over the last few decades the interest in these topics (and more broadly into partial differential equations (PDEs) involving singular integrals and nonlocal operators) has exploded and great attention has been given not only to many important theoretical aspects but also to the application of these mathematical tools and formalisms to describe real-life phenomena in a variety of applications (e.g., \cite{bakunin-2008,bates-2006,gilboa_etal-2008,BuOr-etal14,CuGG18,delcastillonegrete_etal-pp-2012,mainardi-book-2010,metzler_etal-pccp-2014,Plo19,rossikhin_etal-amr-2010,zhang_etal-wrr-2016}). 

From the theoretical standpoint, the method of semigroups offers a simple (and very natural) way to define both positive and negative fractional powers of second order elliptic operators, both in unbounded and bounded domains, within the same mathematical framework. In fact, this definition relies simply on the concept of heat semigroup and on a one-dimensional singular integral involving the fractional exponent (two standard tools very well understood and often used by the mathematical community concerned with the analysis of PDEs). 

Besides providing an amenable platform for the theoretical study of important characteristics of fractional powers of operators, the numerical community has used this approach in the unbounded settings to produce monotone discretizations with optimal error estimates for the fractional Laplacian in $\mathbb{R}^n$.
The first step in this direction was done by Ciaurri et al.~\cite{Cia2015}. In particular, they consider the discrete Laplacian in $\R$ defined by 
$\Delta_h\psi(x)=\frac{\psi(x+h)+\psi(x-h)-2\psi(x)}{h^2}$ and through the analytic expression for $e^{t\Delta_h}$ (explicitly available in dimension one), they showed that, for smooth enough functions $\psi$, $\|(-\Delta_h)^s\psi-(-\Delta)^s\psi\|_{L^\infty(\R)}=O(h^{2-2s})$. Later, the numerical community showed that through the same approach much better (in fact optimal) error estimates could be obtained in any dimension. More precisely, the authors of \cite{dTEnJa18b} showed that $(-\Delta_h)^s$ produces in $\R^n$ monotone finite difference discretizations of $(-\Delta)^s$ with error estimates $O(h^2)$ independently on  the exponent $s$.

Fractional problems involving non-integer powers of elliptic operators on bounded domains like \eqref{EP}-\eqref{BC} have been extensively studied. Regularity issues have been treated for example by Seeley in \cite{See66b, See66,See69} and by several other authors in more recent works (e.g., \cite{Gru16, CaSt16, DiVo12, Vol17}). While the analysis of boundary behaviour for the solution to such problems has been investigated in~\cite{BoFiVa18}, non-homogeneous boundary conditions for these nonlocal operators have been considered in~\cite{AbDu17,antil_etal-arxiv-2017,CuTeGGPa18,lischke_etal-arxiv-2018}. 

The numerical approximation of fractional powers of elliptic operators and the solution to problems in the form of \eqref{EP}-\eqref{BC} have also received substantial attention in recent years. However, perhaps surprisingly, in spite of the successful results obtained in the unbounded context and the fact that the operator definition carries over naturally to the case of a bounded domain, the semigroup formulation was somehow overlooked from the numerical point of view in the bounded settings. This lack of results motivated us to investigate the semigroup approach further from a numerical perspective with the aim of recovering (and possibly improving) approximation estimates already available in the literature for the same type of operator while providing an alternative numerical approach to the solution of fractional elliptic problems on bounded domains.

In~\cite{CuTeGGPa18}, the authors exploit the heat semigroup and the integral definition~\eqref{eq:FLdef2} to derive numerical approximations of the spectral fractional Laplacian applied to a given function, and then use the proposed method in the context of fractional parabolic problems (see \cite{BoFiVa17} for an application of this numerical scheme in the context of fractional porous medium equations). However, the fractional elliptic problem was not considered in~\cite{CuTeGGPa18} and motivated the work presented in this manuscript. 

Elliptic problems involving fractional powers of the Laplacian defined in the spectral sense have been considered by other authors in the literature. For example, a different approach was proposed by Bonito and Pasciak in~\cite{Bon2015}. These authors also use an integral representation of $(-\Delta)^{-s}$, namely,
\begin{equation}\label{eq:invPowerBo}
(-\Delta)^{-s}f(x)=\frac{2\sin(\pi s)}{\pi} \int_{0}^\infty t^{2s-1}(I-t^2\Delta)^{-1}f(x) dt,
\end{equation}
and combine FE approximations for the solution of local elliptic problems of the form $(I-t^2 \Delta)v=f$ with suitable quadrature rules for the considered singular integral. While in the present work we need to compute the semigroup $e^{t\Delta}f$, in~\cite{Bon2015} the authors need to solve a series of elliptic problems $(I-t^2 \Delta)v=f$ for different increasing values of $t$. Nevertheless, the rates of convergence of their approach are comparable to ours when we set $k=1$ and $a=2$ in  Theorem \ref{thm:smooth1}. We mention here also the works \cite{BoPa17a,BoPa17b} in which Bonito and Pasciak extend the above-mentioned approach to a very general class of operators, as well as  the paper by Bonito et al.~\cite{BoWePa17} in which an application of their discretization method is considered in the context of parabolic equations. 

Other important numerical approaches found in the literature include the matrix transfer technique introduced by Ili\'c et al.~in~\cite{IlLiTuAn05,IlLiTuAn06} (which relies on computing an approximation to the fractional power of a matrix representing the discretization of the standard elliptic operator), the work by Song et al.~\cite{song_etal-2017} (based on an efficient method to approximate the eigenpairs of the standard Laplacian on complex geometries),  the articles by Vabishchevich \cite{Vab15,Vab16} (where the solution of (1.1)-(1.2) is viewed and computed as the value of the solution to a suitable pseudo-parabolic problem at a specific point in time),  and the works by Nochetto et al.~\cite{NoOtSa15, Chen-etal15} (relying on the use of a Dirichlet-to-Neumann map to represent the fractional operator). Regarding these last works in fact, it is well-known \cite{CaSi07, CaTa10, StTi10} that for a given $u$, the solution $v(x,y)$ to the extended problem 
\begin{equation}\label{eq:extension}
\begin{split}
\nabla\cdot(y^{1-2s}\nabla v)(x,y) &=0 \qquad \ \textup{for} \quad (x,y) \in \Omega \times (0,\infty),\\
v(x,0)&=u(x)  \quad \textup{for} \quad x \in \Omega,
\end{split}
\end{equation}
is such that $(-\Delta)^su(x)=\lim_{y\to0^+} (-y^{1-2s}v_y(x,y))$. By using this characterization and computing its numerical solution via FE, the authors of~\cite{NoOtSa15} appoximate the fractional elliptic problem \eqref{EP}-\eqref{BC} and obtain error estimates in a suitable weighted Sobolev space.  Similarly to what we do in this work with the heat semigroup $e^{t\Delta_\B}f$,  in~\cite{NoOtSa15} the fast decay of the solution of \eqref{eq:extension} in the extension variable $y\in[0,\infty)$ is exploited to obtain a convenient truncation of the extended unbounded domain.   Numerical methods for optimal control problems related to \eqref{EP}-\eqref{BC} have also been developed using characterization \eqref{eq:extension} in \cite{AnOt15} and \eqref{eq:invPowerBo} in \cite{Do-etal18}.  We refer to the very recent work by Bonito et al.~\cite{Bonito-et-al17} for a general review of numerical methods for fractional diffusion. 

Finally, it is important to mention that the spectral definition of the fractional Laplacian, that is the main focus of the present work, is not the only way to consider fractional operators in bounded domains (e.g., see \cite{MuNa14,delia_etal-cma-2013,bogdan_etal-2003,wrobel-mcs-2019}). A different (and non-equivalent \cite{MuNa14,SeVa14,GaSt18}) way of defining fractional problems in bounded domains is obtained for example by considering the so-called restricted fractional Laplacian. More precisely, given the fractional Laplacian in $\R^n$,
\begin{equation}\label{eq:FLRN}
(-\Delta)^s\psi(x)=\textup{P.V.} \int_{\R^n} \frac{\psi(x)-\psi(y)}{|x-y|^{n+2s}}dy
\end{equation}
one can define the corresponding Dirichlet problem with exterior boundary condition as $(-\Delta)^su(x)=f(x)$ for $x\in \Omega$ and $u(x)=g(x)$ for $x\in \R^n\setminus \Omega$. Numerical approximations in this case are often obtained directly from discretizations of~\eqref{eq:FLRN} either through finite-difference approximations (e.g., \cite{HO2014, HuOb16,Cia2015,dTEnJa18b}) or FE (e.g., \cite{acosta_etal-siam-2015, Ac-etal18}). We refer to the surveys by V\'azquez  \cite{Vaz12,Vaz2014,Vaz17}  for a general overview of the state-of-the-art in nonlinear and fractional diffusion, and to the work by Lischke et al.~\cite{lischke_etal-arxiv-2018} focusing in particular on the fractional Laplacian (reviewing its various definitions and comparing several approximation methods for its computation). We conclude by mentioning also numerical methods for nonlinear and nonlocal evolution problems in $\RN$ that have been treated by several authors, such as  \cite{Dro10,DrJa14,dTe14,CiJa14,dTEnJa18a}.

\noindent{\bf{Organization of the paper.}}
 The assumptions, the quadrature rules for the singular integral, the FE framework, and our main results are given in Section~\ref{sec:main}.  In Section~\ref{sec:quad}, we provide error estimates for our quadrature approximations. The FE estimates as well as the proof of our main results are given in Section~\ref{sec:FE}. Section~\ref{sec:NHBC} concerns nonhomogeneous boundary conditions, while Section \ref{sec:numtests} presents some illustrative numerical experiments supporting our results. Finally, in Section~\ref{sec:comments}, we mention possible extensions and open questions related to the presented work.

\section{Assumptions and main results}\label{sec:main}
Our numerical method is based on the integral representation of the solution to \eqref{EP}-\eqref{BC} given by 
\begin{equation*}
u(x)=(-\Delta_\B)^{-s} f(x):=\frac{1}{\Gamma(s)}\int_0^\infty e^{t\Delta_\B}f(x)\frac{d t}{t^{1-s}},
\end{equation*}
on the approximation of $e^{t\Delta_\B}f$ (solution of the homogeneous heat equation with initial condition $f$) through FE, and on the use of suitable quadrature rules for integrals of the form $\int_0^\infty \rho(t)t^{s-1}dt$.

We will formulate our main results under the following assumptions for the boundary conditions:
\begin{equation}\label{as:B}\tag{A$_{\B}$}
\textup{$\B$ as in \eqref{as:BC}. If $\B$ is Neumann, we ask $\int_\Omega f=0$.}
\end{equation}

Our first main result (Theorem \ref{thm:smooth1}) will be formulated for smooth bounded domains $\Omega \subset \R^n$, $n\geq1$, while the second one (Theorem \ref{thm:convexPol1}) concerns the case of convex polygonal domains ($n=2,3$) where rates of convergence of numerical approximations are limited by the geometry of the domain. 

 \begin{remark}
	In the Neumann case the first eigenvalue of $-\Delta_\B$ is $\lambda_1=0$. By imposing the extra condition $\int_{\Omega}f=0$ one has that $\hat{f}_1=0$ and hence there exists a unique solution of \eqref{EP}-\eqref{BC} such that $\int_{\Omega}u=0$ and $u=\sum_{m=2}^\infty\lambda_m^{-s} \hat{u}_m \varphi_m$. Throughout the manuscript we always let the index $m$ start from one with the understanding that in the Neumann case sums do not include the first term. Hence, only positive eigenvalues really play a role. 
	\end{remark}
	\begin{remark}
	More general boundary conditions could be considered as long as the corresponding operator $-\Delta_\B$ has a discrete spectrum and a symmetric, continuous and coercive bilinear form in the functional space corresponding to the FE setting.
\end{remark}

\subsection{Quadrature of the integral}\label{sec:quadrules}
Let $\rho:[0,T]\to \mathbb{R}$ be  such that $\int_0^T \rho(t)\frac{d t}{t^{1-s}}<\infty$ for some $T>0$ given. Let $\{t_j\}_{j=0}^{N_T}$ be a partition of $[0,T]$ into $N_T$ subintervals of uniform length $\varDelta t$. We define a family of quadrature rules for the integral $\frac{1}{\Gamma(s)}\int_0^T \rho(t)\frac{d t}{t^{1-s}}$ depending on the parameter $s\in(0,1)$ and an additional parameter $r$ (that will be later connected to the regularity of $\rho$) as follows: 
\begin{equation}\label{eq:Qdef}
Q^s_r[\rho]:=\frac{1}{\Gamma(s)}\sum_{j=0}^{N_T}\rho(t_j)~\beta_j,
\end{equation}
where the weights $\beta_j$ are defined as
\begin{equation}\label{eq:betas}
\beta_j:=\left\{ \begin{array}{ll}
~~~ \frac{\varDelta t^s}{s} \times \left\{\begin{array}{ll}
(j+1)^s-j^s, &\qquad \qquad \qquad \qquad j=0,\dots,N_T-1,\\
0 & \qquad \qquad \qquad \qquad j=N_T,
\end{array}\right.& \text{if } r\in[0, 2-2s],\\ 
&  
\\
\frac{\varDelta t^s}{s(1+s)}\times \left\{
\begin{array}{ll}
0& j=0,\\
2^{1+s}-1 & j=1,\\
(j+1)^{1+s}-2j^{1+s}+(j-1)^{1+s}& j=2,\dots,N_T-1,\\
(j-1)^{1+s}+(1+s)j^s-j^{1+s} & j=N_T,\\
\end{array}\right. &\text{if } r\in(2-2s, 2), 	\\		
& \\
\frac{\varDelta t^s}{s(1+s)}\times \left\{
\begin{array}{ll}
1& j=0,\\
(j+1)^{1+s}-2j^{1+s}+(j-1)^{1+s}& j=1,\dots,N_T-1,\\
(j-1)^{1+s}+(1+s)j^s-j^{1+s} & j=N_T,\\
\end{array}\right. & \text{if } r\in [2, 4-2s],   \\ 

\end{array}\right.
\end{equation}
	
\begin{remark}\label{rem:weights}The weights $\beta_j$ are all non-negative by assumption (see Lemma \ref{lem:splitinterp} below) leading to monotone quadrature rules, known to be very robust. Additionally, we point out for later use that there exist positive constants $C_1$ and $C_2$ (independent of $\varDelta t$ and $j$) such that $\beta_0\leq C_1 \varDelta t^s$  and $\beta_j\leq C_2 \int_{t_{j}}^{t_{j+1}}\frac{dt}{t^{1-s}}$ for all $j\geq 0$. 
\end{remark}
We now provide an alternative characterization of $Q^s_r[\rho]$, that will be useful in the proofs of our results. More precisely, \eqref{eq:Qdef} is the integral with respect to the measure $\frac{dt}{t^{1-s}}$ of a suitable interpolant of $\rho$.

\begin{lemma}\label{lem:splitinterp}Let $s\in(0,1)$, $r\in[0,4-2s]$, $T>0$, $\rho:[0,T] \to \R$, $\{t_j\}_{j=0}^{N_T}$ a uniform partition of $[0,T]$ with step $\varDelta t>0$, and
\begin{equation}\label{eq:interpintro}
\mI_r[\rho](t):=\left\{
\begin{array}{ll}
\begin{array}{ll} 
									\rho(t_j),&  \hspace{3.6cm}t\in[t_j,t_{j+1}), \forall j\geq 0,
									\end{array} &\text{if} ~r\in[0,2-2s], \\ 
								& \\
  \left\{ \begin{array}{ll}
									\rho(t_1), &t\in[t_0, t_1),\\
									\frac{t_{j+1}-t}{\varDelta t}\rho(t_j)+\frac{t-t_j}{\varDelta t}\rho(t_{j+1}),& t\in[t_j,t_{j+1}), \forall j \geq 1,\\
									\end{array} \right. & \text{if}~r\in(2-2s,2),\\
								& \\
 \begin{array}{ll} 
										\frac{t_{j+1}-t}{\varDelta t}\rho(t_j)+\frac{t-t_j}{\varDelta t}\rho(t_{j+1}),& \hspace{0.3cm}t\in[t_j,t_{j+1}), \forall j\geq 0, \end{array} & \text{if} ~r\in[2, 4-2s]\\
\end{array}\right.
\end{equation}
Then, 
\begin{equation*}
Q^s_r[\rho]=\frac{1}{\Gamma(s)}\int_0^T \mI_r[\rho](t) \frac{dt}{t^{1-s}}.
\end{equation*}
\end{lemma}	
%The proof of Lemma \ref{lem:splitinterp} will be presented in Section \ref{sec:quad}.

\begin{proof}[Proof of Lemma \ref{lem:splitinterp}]
	For all $t\in[0,T]$ and all $j\geq 0$, let $P^0_j(t):=\mathbb{I}_{[t_j,t_{j+1})}$, that is, the indicator function of the interval $[t_j,t_{j+1})$, and let
	$
	P^1_j(t):=\left(1-\frac{|t-t_j|}{\varDelta t}\right) \mathbb{I}_{[ t_{j-1},t_{j+1}]}.
	$
	The family of interpolants defined in \eqref{eq:interpintro} can then be written as
	$
	\mI_r[\rho](t)= \sum_{j=0}^{N_T}  \rho(t_j) \mathcal{P}_j(t)
	$
	where the functions $\mathcal{P}_j$ are defined as follows:
	\begin{equation}\label{eq:mathcalP}
	\mathcal{P}_j(t):=\left\{
	\begin{array}{ll}
	\begin{array}{ll}
	P^0_j(t),
	& \hspace{1.7cm} j=0,1,\dots,N_T,
	\end{array} &\text{if}~ r\in[0,2-2s],  \\ %\hline
	& \\
	 \left\{	\begin{array}{ll}
	0,
	& j=0,\\
	P^1_0(t)+P^1_1(t),
	& j=1,\\ 
	P^1_j(t),
	& j=2,\dots,N_T,
	\end{array} \right. & \text{if}~r\in(2-2s,2),	\\
	& \\
	  %\hline			
	 \begin{array}{ll}
	P^1_j(t),
	&\hspace{1.7cm} j=0,1,\dots,N_T,
	\end{array} & \text{if}~	r\in[2, 4-2s].\\
	\end{array}\right.
	\end{equation}
Hence, 
	$	\frac{1}{\Gamma(s)}\int_0^T \mI_r[\rho](t) \frac{dt}{t^{1-s}}=\frac{1}{\Gamma(s)}\sum_{j=0}^{N_T} \rho(t_j) \int_0^T \mathcal{P}_j (t) \frac{dt}{t^{1-s}}
	$.
	The integrals $\int_0^T \mathcal{P}_j(t)\frac{dt}{t^{1-s}}$ can be computed explicitly (following for example the strategy in \cite{HO2014})  and with simple calculations one finds that
	$\int_0^T \mathcal{P}_j(t)\frac{dt}{t^{1-s}}=\beta_j$, with $\beta_j$ defined as in \eqref{eq:betas}, thus concluding the proof.
\end{proof}

\begin{remark}
	Note that when $r\in[0,2-2s]$, $\mI_r$ corresponds to the piecewise constant Lagrange interpolant. Similarly, for $r\in[2,4-2s]$, $\mI_r$ is the piecewise linear Lagrange interpolant on $[0,T]$. However, when $r\in (2-2s,2)$, we consider a ``special'' definition of $\mI_r$ at the origin, combining piecewise linear and piecewise constant interpolation, that was chosen in order to obtain optimal convergence results (see Section~\ref{sec:quad}). The choice of $\mI_r[\rho](t)=\rho(t_1)$ for $t\in[t_0,t_1)$ when $r\in(2-2s,2)$ is not the only possible option. In fact, exactly the same error estimates could also be proved by setting $\mI_r[\rho](t)=\rho(t_0)$ for this choice of $t$ and $r$, provided that the corresponding $\mathcal{P}_j(t)$ in \ref{eq:mathcalP} and $\beta_j$ in \ref{eq:betas} are suitably adapted.
\end{remark}

A first approximation of $(-\Delta_\B)^{-s}f(x)$ given by \eqref{eq:IFL} is defined for every $x\in \Omega$ by
\begin{equation}\label{eq:quadratureSolHE2}
	Q^s_r[e^{\cdot \Delta_\B }f(x)]=\frac{1}{\Gamma(s)}\sum_{j=0}^{N_T}e^{t_j\Delta_\B}f(x)~\beta_j.
\end{equation}
As it will be made clear in Section~\ref{sec:mainresu}, the parameter $r$ in the definition of the quadrature weights refers to the regularity of the datum $f$. Note also that for practical reasons we need to truncate the singular integral in \eqref{eq:IFL} at a finite time $T$. However, we will be able to control the remainder using the exponential decay  $e^{t \Delta_\B}f(x)$ as $t\to\infty$.

In general, computation of $e^{t \Delta_\B}f(x)$ cannot be carried out exactly and often (for general domains) an explicit expression for $e^{t \Delta_\B}f(x)$ is not even known. Therefore, suitable numerical approximations have to be introduced and in this manuscript we do so via the FE method (as outlined in Section~\ref{sec:finElem}).
\begin{remark}
Whenever an explicit expression for $e^{t \Delta_\B}f(x)$ is available then \eqref{eq:quadratureSolHE2} provides a suitable numerical approximation with an error estimate depending only on $\Delta t$ (and not on any spatial discretization parameter), as shown by Theorem~\ref{thm:interp}.
\end{remark}

\subsection{Finite elements setting}\label{sec:finElem}
We now introduce some general concepts of the FE theory and the notation we will use to denote the FE approximations in the rest of the manuscript. For a detailed discussion we refer to any classical book on the topic (such as~\cite{Vid97,QV2008}). The weak formulation of \eqref{eq:HeatEq} is given by
\begin{equation}
\label{eq:general}
\left\{
\begin{array}{l}
\dfrac{d}{dt} \langle w(t),v\rangle+a( w(t),v)=0, \quad \forall v \in V, t\in(0,\infty)\\
w(0)=f,
\end{array}
\right.
\end{equation}
where $\langle\cdot,\cdot\rangle$ denotes the $L^2(\Omega)$ inner product, while $a(\cdot,\cdot)$, and $V$ are respectively the bilinear form and the functional space associated to the differential operator $-\Delta$ with boundary conditions $\B$. For homogeneous Dirichlet boundary conditions, $a(w,v)=\int_{\Omega}\nabla w \cdot \nabla v $ and $V=H^1_0(\Omega)$, for Neumann $a(w,v)=\int_{\Omega}\nabla w \cdot \nabla v $ and $V=H^1(\Omega)$, while for Robin $a(w,v)=\int_{\Omega}\nabla w \cdot \nabla v+\int_{\partial \Omega}\kappa w ~v$ and $V=H^1(\Omega)$. The FE approximation of the weak solution $w(x,t)$ is typically obtained by first discretizing the above problem in space and then introducing a suitable temporal discretization scheme. 

Given a bounded domain $\Omega$ and a spatial discretization parameter $h>0$, let $\mathcal{T}_h$ denote the associated domain triangulation. We will always assume that $\mathcal{T}_h$ is a quasi-uniform family of triangulations which is essentially requiring that all elements in the triangulation are of about the same size. 

Let $\{V_h\}$ be a family of finite-dimensional subspaces of $L^2(\Omega)$ and let $\{E_h\}$ be a family of operators $E_h:L^2\to V_h$ approximating the exact solution operator $E$ of the elliptic equation $-\Delta \psi = g$ with boundary conditions $\B(\psi)=0$ (i.e., $Eg=\psi $). We assume that: 
\begin{description}
	\item[i)] $E_h$ is selfadjoint, positive semi-definite on $L^2$, and positive definite on $V_h$;
	\item[ii)] there exists a positive integer $k\geq 1$ such that for any $p$, $2\leq p \leq k+1$,
	\[
	\| (E_h-E)g\|_{L^2} \leq C h^p \|g\|_{H^{p-2}} \quad \text{when } g\in H^{p-2}.
	\]
\end{description} 
Then, the value $k+1$ is known as the order of accuracy of the FE approximation (in space). We denote by $\{\phi_i^k\}_{i=1}^{N_h}$ a basis of $V_h$ so that $N_h$ corresponds to the finite dimension of the space. On polygonal domains, an example of $\{V_h\}$ is given by the set of continuous, piecewise polynomial functions of degree at most $k$ on each element of $\mathcal{T}_h$, while on smooth domains suitable FE modifications are considered to avoid the introduction of extra approximation errors near the boundary.

Let $\varDelta t>0$ be a time step and $t_j=j\varDelta t$ be a uniform grid of time points ($j$ non-negative integer). We denote by $w_{k,\tau}$ the FE solution of~\eqref{eq:general} defined as follows:
\begin{equation*}
\label{eq:FEsol}
w_{k,\tau}(x,t_j)=\sum_{i=1}^{N_h}W_i^{(j)}\phi_i^k(x), \quad \text{for all}\quad x\in \Omega,
\end{equation*}
where $W_i^{(j)}$ is the $i$-th component of the vector $\mathbf{W}^{(j)}$, solution of the iterative linear system
\begin{equation*}
\label{eq:fullydisc}
\begin{array}{l}
(M+\theta \varDelta t A)\mathbf{W}^{(j)}=(M+(\theta-1)\varDelta t A)\mathbf{W}^{(j-1)}\quad \text{for}\quad j\geq 1.
\end{array}
\end{equation*}
Here, $\theta\in[0,1]$ is a parameter used in the temporal discretization scheme, $M$ and $A$ denote the FE mass and stiffness matrices (with entries $M_{i,j}:=\langle \phi_i^k,\phi_j^k \rangle$ and $A_{i,j}:=a(\phi_i^k,\phi_j^k)$, respectively) and the initial condition $\mathbf{W}^{(0)}$ is the vector of coefficients given by the $L^2$-orthogonal projection of the initial condition~$f$ over the FE finite-dimensional space $V_h$.

While the subscript $k$ in the notation $w_{k,\tau}$ is related to the accuracy of the solution in space (FE with polynomials of degree $k$ give an approximation of order $k+1$ in space), the subscript $\tau$ denotes the order of the FE approximation in time. It is known that when $\theta\not=1/2$, we have that $\tau=1$ while for $\theta=1/2$ (Crank-Nicolson scheme), we get $\tau=2$ (see e.g., \cite{iserles-book-2008}). Precise error bounds depending on the FE order in both space and time will be given in Section~\ref{sec:FE}.

Finally, we require the following stability condition:
\begin{equation}\label{as:CFL}\tag{CFL}
\textup{if $\textstyle 0\leq \theta <\frac{1}{2}$, then $\varDelta t\leq Ch^2$ for some $C>0$, depending on $\theta$ and $a(\cdot,\cdot)$}.
\end{equation}
We refer to Chapter 6 in reference~\cite{QV2008} for precise details on this condition and the explicit form of $C$.

\subsection{Main results}
\label{sec:mainresu}
To derive a fully discrete approximation of $(-\Delta_\B)^{-s}f$ for some $f\in \HH^r(\Omega)$, $r\in[0,4-2s]$, we combine the quadrature rule $Q^s_r$ defined in Section \ref{sec:quadrules} with the FE approximation $w_{k,\tau}$ (for some $k\in \N_+$ and $\tau\in\{1,2\}$) to the solution of the heat equation \eqref{eq:HeatEq} with initial condition $f$ presented in Section \ref{sec:finElem}. That is, for all $x\in \Omega$, we define the approximation of $(-\Delta_\B)^{-s}f(x)$ as follows:
\begin{equation*}
Q^s_r[w_{k,\tau}](x) =\frac{1}{\Gamma(s)} \sum_{j=0}^{N_T} w_{k,\tau}(x,t_j) ~ \beta_j.
\end{equation*}

We now present our main results. The first theorem is formulated on smooth bounded domains.

\begin{theorem}\label{thm:smooth1} Let $k\in \N_+$, $n\geq1$, and $\Omega\subset \R^n$ be a bounded domain of class $C^{k+1}$.
Assume \eqref{as:B}, $\kappa\in C^{k}(\Omega)$ in the Robin case, \eqref{as:CFL}, $s\in(0,1)$ and $h,\varDelta t>0$ such that $\varDelta t\sim h^a$ for some $a\in(0,2]$. Assume that 
	\begin{enumerate}
		\item[] {\bf (Regularity)} $f\in \HH^r(\Omega)$ for some $r\in[0,4-2s]$.
		\item[] {\bf(Finite elements)} $w_{k,\tau}$ is the FE solution of \eqref{eq:HeatEq} with $\tau$ chosen as follows:
		\begin{description}
			\item[a)] if $r\in[0,2-2s)$, then either $\tau=1$ or $\tau=2$;
			\item[b)] if $r\in[2-2s,4-2s)$, then $\tau=2$.
			\end{description}
		\item[] {\bf{(Quadrature)}} $Q^s_r$ is given by \eqref{eq:Qdef} with $T> \frac{1}{\lambda_{\min}}\log \left(\frac{1}{\varDelta t^{\frac{r}{2}+s}}\right)$ and $\lambda_{\min}$ the first positive eigenvalue of $-\Delta_\B$ on $\Omega$.
	\end{enumerate}
	Then, there exists a constant $C>0$ such that
	\begin{equation*}\label{eq:smoothapprox}
	\|Q^s_r[w_{k,\tau}]-(-\Delta_\B)^{-s}f\|_{L^2(\Omega)}\leq C_{r,h} ~h^{\min\{k+1,a(\frac{r}{2}+s)\}}~ \|f\|_{\HH^r(\Omega)},
	\end{equation*}
	with
		\[
		C_{r,h}:=\left\{ 
		\begin{array}{ll}
		C& r\notin \{k+1-2s,4-2s\}\\
		C~\log(h^{-a})& r\in \{k+1-2s,4-2s\}.
		\end{array}
		\right.
		\]
	
\end{theorem}

We stress that the bound $h^{k+1}$ in the accuracy of the method is inherited from the order of convergence of the FE approximation, while $h^{a(r/2+s)}$ comes from our quadrature.

Our second result concerns convex polygonal domains in dimension two and three. We note that in the one-dimensional case, every convex domain is of class $C^\infty$ and thus the previous result holds.

\begin{theorem}\label{thm:convexPol1}
Let $n\in\{2,3\}$, $\Omega$ be a convex polytope in $\R^n$, and $k=1$. Then the result of Theorem~\ref{thm:smooth1} still holds.
\end{theorem}

The proofs of Theorem~\ref{thm:smooth1} and Theorem~\ref{thm:convexPol1} are the main focus of the paper and will be a consequence of the results presented in Sections \ref{sec:quad} and \ref{sec:FE}. 

\subsection{Some comments on the computational cost of our approach}
\label{sec:computcost}
The computational cost of our algorithm to produce the numerical approximation of $\DBm f$ heavily relies on the number of points $N_T$ needed by the quadrature $Q_r^s$ which depends on $\varDelta t$ as follows: 
\[
N_T>\frac{\frac{r}{2}+s}{\lambda_{\min}}~ \varDelta t^{-1} \log(\varDelta t^{-1}).
\] 
The time step $\varDelta t$ is in turn chosen proportional to a suitable power of $h$, which was introduced to reduce the overall computational cost of our method, whenever possible.

When $0\leq \theta<\frac{1}{2}$, the \eqref{as:CFL} condition forces us to select $\varDelta t\sim h^2$, while for $\theta\geq \frac{1}{2}$ this restriction is not imposed. Nevertheless, for $r\in[0,2-2s)$, even if we select $\theta=\frac{1}{2}$ (and hence $\tau=2$) the best error decay is still obtained with $a=2$ (i.e., $\varDelta t\sim h^2$) since $r+2s<2$ and $k+1\geq 2$ for all $k\in\mathbb{N}_+$.

On the other hand, if $r\in(2-2s,4-2s]$ and linear FE are used ($k=1$), the error decay of our approximation is always bounded by $h^2$. This allows us to select an ``optimal'' value of $a$, namely, $a_{\rm opt}:=\frac{2}{r/2+s}\leq 2$, to reduce computational costs while preserving quadratic convergence. Note that, for $s$ fixed, $a_{\rm opt}$ decreases as $r$ increases. Therefore, as the regularity of the right-hand side of the fractional elliptic problem improves, the solution approximation can be obtained with lower computational effort (i.e., a larger $\varDelta t$ can be selected). Similarly, for $r$ fixed, $a_{\rm opt}$ decreases as the fractional parameter $s$ increases (that is, the closer \eqref{EP}-\eqref{BC} is to the local elliptic problem). 

An example worth mentioning is the case $r=4-2s$ for which $a_{\rm opt}=1$, meaning that the choice $\varDelta t\sim h$ already gives quadratic order of convergence. This and other interesting examples are presented in Table~\ref{tab:decay} below. We refer to Section~\ref{sec:numtests} and Table~\ref{tab:cost} for a simple numerical test showing evidence on the reduction of computational cost in one of these cases. Analogous considerations could be made for general $k\in \N_+$.

 \begin{table}[h!]
 	\centering
 		\begin{tabular}{|c||c|c|c|c|c|c|c|c|c|}\hline
 			$r$ & \multicolumn{2}{c|}{0}& \multicolumn{2}{c|}{$2-2s$}& \multicolumn{3}{c|}{$3-2s$}&\multicolumn{2}{c|}{$4-2s$\rule{0pt}{2.5ex}} \\ \hline
 			$k$& \multicolumn{2}{c|}{1}& \multicolumn{2}{c|}{1} &\multicolumn{2}{c|}{1}&2 &1&3\rule{0pt}{2.5ex} \\ \hline
 			$\tau$& \multicolumn{2}{c|}{1}& \multicolumn{2}{c|}{2} & \multicolumn{3}{c|}{2} & \multicolumn{2}{c|}{2\rule{0pt}{2.5ex}}  \\  \hline
 			$a$& 1& 2& 1 & 2 &1& 4/3& 2 &1 & 2\rule{0pt}{2.5ex} \\  \hline\hline
 			EDh& $h^s$ & $h^{2s}$ & $h$ & $h^2$ &$h^{3/2}$ &$h^2$ & $h^3$& $h^2$ & $h^4$\rule{0pt}{3ex} \\ \hline
 		\end{tabular}
 	\caption{Error dependence on $h$ (EDh) as stated in Theorem \ref{thm:smooth1} (logarithmic term omitted), for data $f\in\HH^r(\Omega)$ with different combinations of the parameters $r$, $k$, $\tau$, and $a$.}
 	\label{tab:decay}
 \end{table}

\section{Quadrature estimates}\label{sec:quad}
For every $x\in \Omega$ we apply the family of quadratures $Q_r^s$ to the solution of the heat equation given by $e^{t \Delta_\B }f(x)$. We recall the spectral decomposition 
\begin{equation}\label{eq:specdes2}
e^{t\Delta_\B}f(x)=\sum_{m=1}^\infty e^{-\lambda_m t}\hat{f}_m \varphi_m(x)
\end{equation}
and note that  $
\mI_r[e^{\cdot\Delta_\B}f(x)](t) =\sum_{m=1}^\infty \mI_r[e^{-\lambda_m \cdot}](t)\hat{f}_m\varphi_m(x)$ 
(see proof of Theorem \ref{thm:interp}). First, in Lemma \ref{lem:interp} we will provide estimates for the error introduced by $\mI_r[e^{-\lambda \cdot}]$ for every $\lambda>0$. Later, in Theorem \ref{thm:interp} we use the previous relation to bound the total error $(-\Delta_\B)^{-s}f- Q^s_r[e^{\cdot \Delta_\B }f]$ in $L^2(\Omega)$.

To prove Lemma \ref{lem:interp} we will use the following technical result.

\begin{lemma}\label{prop:prelim}
	Assume $s\in(0,1)$, $\nu\in(-s,1-s]$, $\lambda,\varDelta t>0$, $t_j:=j\varDelta t$ for all $j\in\N$ and $\phi(t):=e^{-\lambda t}$. Then there exists a constant $C=C(s,\nu)>0$ such that the following estimate holds for all $j\geq0$ if $\nu\in[0,1-s]$ and for all $j\geq1$ if $\nu\in(-s,0)$: 
	\[
	\int_{t_j}^{t_{j+1}}\phi(\eta)\left(\int_{\eta}^{t_{j+1}}\frac{dt}{t^{1-s}}\right)~d\eta \leq C  \varDelta t^{s+\nu} \int_{t_j}^{t_{j+1}}\frac{\phi(\eta)}{\eta^\nu}d\eta
	\]
\end{lemma}

\begin{proof} We divide the proof in two steps.
	\textbf{1.}~If $\nu\geq0$, then $\eta^\nu\leq t^\nu$ and $\forall j\geq 0$
	\[
	\int_{t_j}^{t_{j+1}}\phi(\eta)~\frac{\eta^\nu}{\eta^\nu}\left(\int_{\eta}^{t_{j+1}}\frac{dt}{t^{1-s}}\right)d\eta \leq \int_{t_j}^{t_{j+1}}\frac{\phi(\eta)}{\eta^\nu}\left(\int_{\eta}^{t_{j+1}}\frac{dt}{t^{1-(s+\nu)}}\right)d\eta = C \int_{t_j}^{t_{j+1}}\frac{\phi(\eta)}{\eta^\nu}\left( t_{j+1}^{s+\nu}-\eta^{s+\nu} \right)d\eta.
	\]
	When $\nu<0$ and $j\geq1$, for any $\eta,t\in[t_j,t_{j+1}]$ we have $t/\eta\leq(j+1)/j\leq2$. Hence, $(t/\eta)^{-\nu}\leq 2^{-\nu}$, that is, $\eta^\nu\leq \frac{t^\nu}{2^v}$, and the estimate above holds (for a different constant $C$).
	
	\textbf{2.}~Note first that in all considered cases $0<s+\nu\leq 1$. If $j=0$, then by definition $t_{j+1}^{s+\nu}-\eta^{s+\nu}=\varDelta t^{s+\nu}-\eta^{s+\nu}\leq \varDelta t^{s+\nu}$. For $j\geq1$, we use the fact that the function $\psi(\eta):=\eta^{s+\nu}$ is concave. Hence, since $\eta \geq t_1= \varDelta t$, we obtain
	\[
	t_{j+1}^{s+\nu}-\eta^{s+\nu}\leq \psi'(\eta)(t_{j+1}-\eta)=(s+\nu)\eta^{s+\nu-1}(t_{j+1}-\eta)\leq (s+\nu) \varDelta t^{s+\nu-1}\varDelta t=\varDelta t^{s+\nu}.
	\]
	Combining this bound with the one of step \textbf{1.} gives the desired result.
\end{proof}

We now estimate the error introduced by  $\mI_r[e^{-\lambda \cdot}]$.

\begin{lemma}\label{lem:interp}
	Assume $s\in(0,1)$, $r\in [0,4-2s] $ and $\varDelta t, T, \lambda>0$. Let $\phi(t):=e^{-\lambda t}$ and $\mI_r$ be given by~\eqref{eq:interpintro}. Then there exists a constant $C=C(r,s)>0$ such that
\[
	\left| \int_0^T \left(\mI_r[\phi](t)-\phi(t) \right)\frac{d t}{t^{1-s}}\right|\leq C \lambda^{\frac{r}{2}}\varDelta t^{s+\frac{r}{2}}. 
\]
\end{lemma}

%We are now ready to prove Lemma \ref{lem:interp}.
\begin{proof}[Proof of Lemma \ref{lem:interp}] We divide the proof in three steps, considering three different ranges for $r$.\\
	\textbf{1.}~Let $r\in[0,2-2s]$ and define
	\begin{equation*}
	I:=\left| \int_0^T \left(\mI_r[\phi](t)-\phi(t) \right)\frac{d t}{t^{1-s}}\right|= \left|\sum_{j=0}^{N_T} \int_{t_j}^{t_{j+1}} \left(\phi(t_j)-\phi(t) \right)\frac{d t}{t^{1-s}}\right|=\left|\sum_{ j=0}^{N_T} \int_{t_j}^{t_{j+1}} \left(-\int_{t_j}^t \phi'(\eta) d \eta \right)\frac{d t}{t^{1-s}}\right|.
	\end{equation*}
	Noting that $\phi'(t)=-\lambda \phi(t)$ and that all integrands above are positive, we can remove the absolute value and rewrite $I$ as
	\begin{equation*}
	I=\lambda \sum_{j=0}^{N_T} \int_{t_j}^{t_{j+1}} \left(\int_{t_j}^t \phi(\eta) d \eta\right) \frac{d t}{t^{1-s}}= \lambda \sum_{j=0}^{N_T} \int_{t_j}^{t_{j+1}} \phi(\eta) \left(\int_{\eta}^{t_{j+1}}  \frac{d t}{t^{1-s}}\right) d\eta.
	\end{equation*}
		Since $0< \frac{r}{2}\leq1-s$, we apply Lemma~\ref{prop:prelim} with $\nu=\frac{r}{2}$ for all $j\geq0$ together with \eqref{eq:numPower} to obtain
		\begin{equation*}
		\begin{split}
		I\leq \lambda \varDelta t^{s+\frac{r}{2}} \sum_{j=0}^{N_T} \int_{t_j}^{t_{j+1}}\frac{\phi(\eta)}{\eta^\frac{r}{2}}d\eta \leq C\lambda\varDelta t^{s+\frac{r}{2}} \int_0^\infty\frac{\phi(\eta)}{\eta^{\frac{r}{2}}} d \eta \leq C\lambda\varDelta t^{s+\frac{r}{2}} \Gamma\left(1-\frac{r}{2}\right) \lambda^{-1+\frac{r}{2}}= C\varDelta t^{s+\frac{r}{2}}  \lambda^{\frac{r}{2}}.
		\end{split}
		\end{equation*}
	
	\textbf{2.}~If $r\in(2-2s,2)$ then
	\[
	I \leq \left|\int_0^{\Delta t} \left(\phi(t_1)-\phi(t)\right)\frac{d t}{t^{1-s}}\right|+ \left| \sum_{j=1}^{N_T}  \int_{t_{j}}^{t_{j+1}}  (\mI_r[\phi](t)-\phi(t)) \frac{d t}{t^{1-s}}\right|=:I_1+I_2.
	\]
	First we bound $I_1$. Note that $1-e^{-\xi}\leq \min\{\xi,1\}\leq \xi^{\frac{r}{2}}$ for all $\xi\geq0$ and $r\in(0,2]$. Thus, for $t\in[0,t_1]$,
		\[
		0\leq\phi(t)-\phi(t_1)=\phi(t)(1-\phi(t_1-t))\leq \phi(t) (1-\phi(t_1))\leq \phi(t)~(\lambda t_1)^{\frac{r}{2}} \leq(\lambda t_1)^{\frac{r}{2}},
		\]
		so that
$
	I_1\leq \varDelta t^{\frac{r}{2}}\lambda^{\frac{r}{2}}\int_0^{\Delta t} \frac{d t}{t^{1-s}}\leq C\varDelta t^{s+\frac{r}{2}}  \lambda^{\frac{r}{2}}. 
$

We now bound $I_2$. For $t\in[t_j,t_{j+1})$ we use the Taylor expansion with explicit reminder and the explicit form of $\mI_r$ given by \eqref{eq:interpintro} to get
	\begin{equation*}
	\begin{split}
	\mI_r[\phi](t)-\phi(t)&=- \frac{(t-t_j)(t_{j+1}-t)}{\Delta t} \int_0^1 \Big[\phi'(t_j(1-\xi)+\xi t)-\phi'(t_{j+1}(1-\xi)+\xi t)\Big] d\xi \\
	&= \lambda \frac{(t-t_j)(t_{j+1}-t)}{\Delta t}\int_0^1 \phi(t_j(1-\xi)+\xi t) \Big[1-\phi(\Delta t(1-\xi))\Big]d \xi.
	\end{split}
	\end{equation*}
	Observe that for all $\xi\in[0,1]$, $1-\phi(\Delta t(1-\xi))\leq 1-\phi(\varDelta t) \leq \lambda \varDelta t$. Then
		\begin{equation}\label{eq:auxest3}
	\begin{split}
	\mI_r[\phi](t)-\phi(t)&\leq \lambda^2 (t-t_j)(t_{j+1}-t) \int_0^1 \phi(t_j(1-\xi)+\xi t) ~d\xi\\
	&= \lambda^2 (t-t_j)(t_{j+1}-t)\int_{t_j}^{t} \phi(\eta) \frac{d \eta}{t-t_j}\leq \varDelta t \lambda^2 \int_{t_j}^{t} \phi(\eta) ~d \eta.
	\end{split}
	\end{equation}
Combining estimate \eqref{eq:auxest3}, Lemma \ref{lem:interp} with $\nu =\frac{r}{2}-1 \in (-s,0)$ for all $j\geq1$, and \eqref{eq:numPower} we get
\begin{equation*}
\begin{split}
		I_2&\leq  \varDelta t \lambda^2 \sum_{j=1}^{N_T}  \int_{t_{j}}^{t_{j+1}} \left(\int_{t_j}^{t} \phi(\eta) ~d \eta \right)\frac{d t}{t^{1-s}}= \varDelta t \lambda^2 \sum_{j=1}^{N_T}  \int_{t_{j}}^{t_{j+1}} \phi(\eta)\left(\int_{\eta}^{t_{j+1}}  \frac{d t}{t^{1-s}}\right)d\eta\\
		&\leq  C \varDelta t \lambda^2 \varDelta t^{s+\frac{r}{2}-1} \int_0^\infty \frac{\phi(\eta)}{\eta^{\frac{r}{2}-1}}d\eta= C  \varDelta t^{s+\frac{r}{2}}\lambda^2~ \Gamma\left(2-\frac{r}{2}\right)\lambda^{-2+\frac{r}{2}}=C \varDelta t^{s+\frac{r}{2}} \lambda^{\frac{r}{2}}
\end{split}
\end{equation*}
		\textbf{3.}~Let $r\in[2,4-2s]$. Note that the bound \eqref{eq:auxest3} also holds for $j=0$ and hence in this case we can apply Lemma \ref{prop:prelim} with $\nu=\frac{r}{2}-1\in[0,1-s]$ for all $j\geq0$. Then, once again,
		\[
		\left| \int_0^T \left(\mI_r[\phi](t)-\phi(t) \right)\frac{d t}{t^{1-s}}\right|\leq \varDelta t \lambda^2 \sum_{j=0}^{N_T}  \int_{t_{j}}^{t_{j+1}} \phi(\eta) \left(\int_{\eta}^{t_{j+1}} \frac{d t}{t^{1-s}}\right)d\eta \leq 
		C \varDelta t^{s+\frac{r}{2}} \lambda^{\frac{r}{2}} . 
		\]
\end{proof}

\subsection{Quadrature error on the heat semigroup}
We provide now a bound for the approximation error obtained when our quadrature is applied to the heat semigroup. This bound depends on the regularity of the initial condition $f$ and the proof of this result exploits the asymptotic behaviour of the heat semigroup, approaching zero exponentially fast.

\begin{theorem}\label{thm:interp}
Assume $n\geq1$, $s\in (0,1)$ and \eqref{as:B} and $\Omega$ be bounded subset of $\R^n$. Let $f\in \HH^r(\Omega)$ for some $r\in[0,4-2s]$ and $T> \frac{1}{\lambda_{\min}}\log \left(\frac{1}{\varDelta t^{\frac{r}{2}+s}}\right)$, with $\lambda_{\min}$ the first positive eigenvalue of $(-\Delta_\B)$ on~$\Omega$. Then
\[
\|(-\Delta_\B)^{-s}f- Q^s_r[e^{\cdot \Delta_\B }f]\|_{L^2(\Omega)}\leq C \|f\|_{\HH^r(\Omega)} \varDelta t^{\frac{r}{2}+s}.
\]
where $C=C(r,s,\Omega,\B)>0$.
\end{theorem}

\begin{proof}[Proof of Theorem \ref{thm:interp}]
First we consider
\[
(-\Delta_\B)^{-s}f(x)= \frac{1}{\Gamma(s)} \int_0^T e^{t\Delta_\B}f(x)\frac{d t}{t^{1-s}}+\frac{1}{\Gamma(s)}\int_T^\infty e^{t\Delta_\B}f(x)\frac{d t}{t^{1-s}}=:J_1(x)+J_2(x).
\]
We proceed in two steps: first we bound $\|J_1-Q^s_r[e^{\cdot \Delta_\B }f]\|_0$ and then $\|J_2\|_0$ which allows us to conclude the proof via the triangle inequality. \\
\textbf{1.}~As shown in the proof of Lemma~\ref{lem:splitinterp}, the interpolant defined in \eqref{eq:interpintro} can be written as
$
\mI_r[\rho](t)= \sum_{j=0}^{N_T}  \rho(t_j) \mathcal{P}_j(t)
$
with $\mathcal{P}_j$ given by \eqref{eq:mathcalP}. Hence, the spectral decomposition \eqref{eq:specdes2} allows us to write
\begin{equation*}
\begin{split}
\mI_r[e^{\cdot\Delta_\B}f(x)](t) &= \sum_{j=0}^{N_T}  e^{t_j\Delta_\B}f(x) \mathcal{P}_j(t)=\sum_{j=0}^{N_T} \left(\sum_{m=1}^\infty e^{-\lambda_m t_j} \hat{f}_m\varphi_m(x)\right)\mathcal{P}_j(t)\\
&= \sum_{m=1}^\infty \left(\sum_{j=0}^{N_T}e^{-\lambda_m t_j}\mathcal{P}_j(t)\right) \hat{f}_m\varphi_m(x)=\sum_{m=1}^\infty \mI_r[e^{-\lambda_m \cdot}](t)\hat{f}_m\varphi_m(x).
\end{split}
\end{equation*}
Therefore,  
\begin{equation*}
\begin{split}
J_1(x)-Q^s_r[e^{\cdot \Delta_\B}f(x)]&=\frac{1}{\Gamma(s)}\int_0^T \left[e^{t\Delta_\B}f(x)- \mI_r[e^{\cdot \Delta_\B }f(x)](t)\right] \frac{d t}{t^{1-s}}\\
&=\frac{1}{\Gamma(s)}\int_0^T \left[\sum_{m=1}^\infty \left(e^{-\lambda_m t} - \mI_r[e^{-\lambda_m \cdot}](t)\right) \hat{f}_m\varphi_m(x)\right] \frac{d t}{t^{1-s}}\\
&=\frac{1}{\Gamma(s)}\sum_{m=1}^\infty \hat{f}_m \varphi_m(x) \int_0^T (e^{-\lambda_m t}-\mI_r[e^{-\lambda_m \cdot}](t) )\frac{d t}{t^{1-s}}
\end{split}
\end{equation*}
The above identity, together with Lemma \ref{lem:interp}, shows that
\begin{equation*}
\begin{split}
\|J_1-Q^s_r[e^{\cdot \Delta_\B }f]\|_0^2&=\frac{1}{\Gamma(s)^2} \sum_{m=1}^\infty |\hat{f}_m|^2 \left( \int_0^T (e^{-\lambda_m t}-\mI_r[e^{-\lambda_m \cdot}](t) )\frac{d t}{t^{1-s}}\right)^2 \\
&\leq  C \varDelta t^{r+2s}  \sum_{m=1}^\infty |\hat{f}_m|^2 \lambda_m^r=C \varDelta t^{r+2s} \|f\|_r^2
\end{split}
\end{equation*}

\textbf{2.}~Bound for $\|J_2\|_0$. Since $r\geq0$, then $f\in L^2(\Omega)$ and $\|f\|_0\leq C \|f\|_r$ for some positive constant $C=C(\Omega,r)$. Hence, the expression of $e^{t\Delta_\B}f$ given by \eqref{eq:Wseries} yields
\[
\|e^{t\Delta_\B}f\|_0=\left(\sum_{m=1}^\infty e^{-2\lambda_m t} |\hat{f}_m|^2\right)^{\frac{1}{2}}\leq e^{-\lambda_{\min}t}\left(\sum_{m=1}^\infty |\hat{f}_m|^2\right)^{\frac{1}{2}}=e^{-\lambda_{\min}t}\|f\|_0\leq Ce^{-\lambda_{\min}t}\|f\|_r
\]
This, together with Minkowski integral inequality, implies
\begin{equation*}
\begin{split}
\|J_2\|_{0}&\leq \frac{1}{\Gamma(s)}\int_T^\infty \|e^{t\Delta_\B}f\|_0\frac{d t}{t^{1-s}}\leq C \|f\|_r \int_T^\infty e^{-\lambda_{\min}t} \frac{d t}{t^{1-s}}\leq  C \|f\|_r  e^{-\lambda_{\min}T}\leq C \|f\|_r \varDelta t^{\frac{r}{2}+s}.
\end{split}
\end{equation*}
where the last inequality follows from the fact that $T>\frac{1}{\lambda_{\min}}\log \left(\frac{1}{\varDelta t^{\frac{r}{2}+s}}\right)$. The proof concludes by combining the bounds above and the fact that 
\[
\|(-\Delta_\B)^{-s}f- Q^s_r[e^{\cdot \Delta_\B }f]\|_0=\|J_1- Q^s_r[e^{\cdot \Delta_\B }f]+J_2\|_0\leq  \|J_1-Q^s_r[e^{\cdot \Delta_\B}f]\|_0+\|J_2\|_0.
\]
\end{proof}

\section{Finite elements estimates and proof of the main result}\label{sec:FE}
Let $\Omega$ be a smooth domain and let $w_{k,\tau}$ be the FE approximation of the heat semigroup defined in Section~\ref{sec:finElem}. Then, classical FE estimates for the $L^2$-error in the semi-discrete and fully discrete solution of~\eqref{eq:general} (the results of Theorems 3.5 and 7.3 in~\cite{Vid97}) can be combined to give the following bound depending on $\varDelta t$, $h$, the regularity of $f$, and $t_j$ with $j\geq 1$:
\begin{equation}\label{eq:errorFEVidar}
\|w_{k,\tau}(\cdot, t_j)-e^{t_j \Delta_\B}f(\cdot)\|_0\leq C_1 h^{l}~t_j^{-\frac{l-p}{2}}\|f\|_p + C_2 \varDelta t^{\llll}~ t_j^{-(\llll-\tilde{p})}\|f\|_{2\tilde{p}} \quad \textup{for any} \quad \left\{\begin{split}
&0\leq p \leq l \leq k+1\\
&0\leq \pp \leq \llll \leq \tau.
\end{split}\right.
\end{equation}

However, in non-smooth domains the above estimate does not hold in full generality because of limited elliptic regularity~\cite{grisvard_book-1985}, which essentially prevents obtaining high-order FE approximations in space near the domain boundary. Nevertheless, for convex polygonal domains the bound~\eqref{eq:errorFEVidar} still holds if we restrict $k$ to be $k=1$, that is, provided that linear FE are considered. 

Through \eqref{eq:errorFEVidar} we are able to prove the following result for smooth domains:
\begin{proposition}\label{thm:FE2} Under the assumptions and notation of Theorem \ref{thm:smooth1} we have that 
\[
\|Q^s_r[w_{k,\tau}]-Q^s_r[e^{\cdot \Delta_\B }f]\|_0\leq C_{r,h} ~ h^{\min\{k+1, a(\frac{r}{2}+s)\}} \|f\|_r.
\]
\end{proposition}

\begin{proof}
Here, to simplify the notation we will use the symbol $a\wedge b:=\min\{a,b\}$. Let $N_1=\lceil \varDelta t^{-1} \rceil$ so that $t_j<1$ for all $j< N_1$, and let $f_h:=w_{k,\tau}(\cdot,0)$ be the $L^2(\Omega)$-orthogonal projection of $f$ onto the FE space $V_h$ introduced in Section~\ref{sec:finElem}. Then
\begin{equation}\label{eq:splitsum}
\begin{split}
Q^s_r&[w_{k,\tau}(x,\cdot)]-Q^s_r[e^{\cdot \Delta_\B }f(x)]= \frac{1}{\Gamma(s)} \sum_{j=0}^{N_T}\left(w_{k,\tau}(x,t_j)-e^{t_j\Delta_\B}f(x)\right)\beta_j\\
&= \frac{1}{\Gamma(s)}\left((f_h-f)\beta_0+ \sum_{j=1}^{N_1-1} (w_{k,\tau}(x,t_j)-e^{t_j\Delta_\B}f(x))\beta_j + \sum_{j=N_1}^{N_T} (w_{k,\tau}(x,t_j)-e^{t_j\Delta_\B}f(x))\beta_j\right).
\end{split}
\end{equation}
In order to bound the norm $\|Q^s_r[w_{k,\tau}]-Q^s_r[e^{\cdot \Delta_\B }f]\|_0$, we bound separately the three terms above.

\textbf{1.}~$f\in \HH^r(\Omega)$ imply $\|f_h-f\|_0\leq C h^{(k+1)\wedge r} \|f\|_r$. Since $\beta_0\leq C \varDelta t^s$ and $\varDelta t \sim h^a$ for some $a\in(0,2]$, then
\[
 \|f-f_h\|_0~\beta_0\leq C~\|f\|_r~ h^{(k+1)\wedge r}  \varDelta t^s=C~\|f\|_r~h^{(k+1+as) \wedge (r+as)}\leq C~\|f\|_r~h^{(k+1) \wedge a(\frac{r}{2}+s)}. 
\]

\textbf{2.}~For the second term in \eqref{eq:splitsum} we have that 
\[
I:=\left\|\sum_{j=1}^{N_1-1} (w_{k,\tau}(\cdot,t_j)-e^{t_j\Delta_\B}f(\cdot))\beta_j\right\|_0 \leq \sum_{j=1}^{N_1-1} \|w_{k,\tau}(\cdot,t_j)-e^{t_j\Delta_\B}f(\cdot)\|_0~\beta_j.
\]
Let $r\in[0,\min\{k+1,4\}-2s)$. We will use estimate~\eqref{eq:errorFEVidar} with a precise choice of the parameters $p,\pp,l$ and $\llll$. Since $r+2s<(k+1)\wedge 4$, there exists an $\epsilon>0$ such that $r+2s+\epsilon\leq (k+1)\wedge 4$. Hence, we set $p=r$, $\pp=r/2$, $l=r+2s+\epsilon$ and $\llll=r/2+s+\epsilon/2$. Clearly, $0\leq p\leq l\leq k+1$ and $0\leq\pp \leq \llll \leq \tau$ ($\tau=1$ is admissible if $r\in[2-2s)$ but $\tau=2$ is necessary if $r\geq2-2s$). We use now \eqref{eq:errorFEVidar}, the bound on the $\beta_j$ given by Remark~\ref{rem:weights}, and the fact that for all $j\geq1$ we have $1/t_j\leq 2 /t_{j+1}$. Then
\begin{equation*}
\begin{split}
\sum_{j=1}^{N_1-1} \|w_{k,\tau}(\cdot,t_j)-e^{t_j\Delta_\B}f(\cdot)\|_0~\beta_j &\leq C\|f\|_r (h^{r+2s+\epsilon}+ \varDelta t^{\frac{r}{2}+s+\frac{\epsilon}{2}}) \sum_{j=1}^{N_1-1} \frac{1}{t_{j+1}^{s+\frac{\epsilon}{2}}}\int_{t_{j}}^{t_{j+1}}\frac{dt}{t^{1-s}}\\
&\leq C\|f\|_r (h^{r+2s+\epsilon}+ \varDelta t^{\frac{r}{2}+s+\frac{\epsilon}{2}}) \int_{\varDelta t}^{1}\frac{dt}{t^{1+\frac{  \epsilon}{2}}}\\
&\leq C\|f\|_r (h^{r+2s+\epsilon}+ \varDelta t^{\frac{r}{2}+s+\frac{\epsilon}{2}}) \varDelta t^{-\frac{  \epsilon}{2}}
%&= C\|f\|_r (h^{r+2s+\epsilon} \varDelta t^{-\epsilon/2}+ \varDelta t^{\frac{r}{2}+s}).
\end{split}
\end{equation*}
The relation $\varDelta t\sim h^a$ for $a\in(0,2]$ implies $(1-\frac{a}{2})\epsilon>0$ and $a(\frac{r}{2}+s)<r+2s$, so that $I\leq C\|f\|_r h^{a(\frac{r}{2}+s)}.$

If $k\geq 3$, then only the case $r=4-2s$ was not covered via the previous approach. Now let either $r=k+1-2s$ with $k=1,2$ or $r=4-2s$ for any $k\geq 1$. This time the error bound involves a logarithmic correction factor. In fact, in this case the best parameter choice is obtained by setting $p=r$, $l=k+1 $, $\pp=r/2$, and $\llll=r/2+s$, leading to
\begin{equation*}
\begin{split}
I&\leq C\|f\|_r (h^{k+1}+ \varDelta t^{\frac{r}{2}+s}) \int_{\varDelta t}^{1}\frac{dt}{t}\leq C\|f\|_r ~h^{(k+1) \wedge a(\frac{r}{2}+s)}\log(h^{-a}). 
\end{split}
\end{equation*}
This method allows us to cover the entire range of possible values for $r$ with $k\geq 3$.

Finally, let $r\in(k+1-2s,4-2s)$ with either $k=1$ or $k=2$. There exist $\epsilon,\delta >0$ such that $r\geq k+1-2s+\epsilon$ with $\epsilon \leq 2s$, and $\frac{r}{2}+s+\delta\leq 2=\tau$. Thus, we can apply \eqref{eq:errorFEVidar} with $p=k+1-2s+\epsilon$, $l=k+1$, $\pp=r/2$, and $\llll=r/2+s+\delta$, and combine it with the bound on the $\beta_j$ to obtain
\begin{equation*}
\begin{split}
I&\leq C (h^{k+1}\|f\|_{k+1-2s+\epsilon} \int_{\varDelta t}^{1}\frac{dt}{t^{1-\frac{\epsilon}{2}}} + \varDelta t^{\frac{r}{2}+s+\delta}\|f\|_{r} \int_{\varDelta t}^{1}\frac{dt}{t^{1+\delta}} ) \leq C (h^{k+1}+ \varDelta t^{\frac{r}{2}+s})\|f\|_{r} . 
\end{split}
\end{equation*}

\textbf{3.}~We now bound the norm of the third term in \eqref{eq:splitsum}. For this we apply the triangle inequality, we exploit \eqref{eq:errorFEVidar}, and use that $1/t_j<1$ for all $j\geq N_1$. This time the simple choice $p=\pp=0$, $l=k+1$, $\llll=\tau$, allows us to conclude
\begin{equation*}
\begin{split}
\sum_{j=N_1}^{N_T} \|w_{k,\tau}(\cdot,t_j)-e^{t_j\Delta_\B}f(\cdot)\|_0~\beta_j &\leq C\|f\|_0 (h^{k+1}+ \varDelta t^\tau) \int_{1}^{\infty}\frac{dt}{t^{2-s}}\\
&\leq C\|f\|_r (h^{k+1}+ \varDelta t^\tau)\leq C\|f\|_r~h^{(k+1) \wedge a(\frac{r}{2}+s)},
\end{split}
\end{equation*}
where the last inequality comes from the fact that $\tau$ is suitably chosen for different ranges of $r$. 

\end{proof}
Due to the restrictions on $k$, the result on convex polygonal domains reads as follows.
\begin{proposition}\label{thm:FE3} Under the assumptions and notation of Theorem \ref{thm:convexPol1} we have that 
\[
\|Q^s_r[w_{1,\tau}]-Q^s_r[e^{\cdot \Delta_\B }f]\|_0\leq C_{r,h} ~ h^{\min\{2, a(\frac{r}{2}+s)\}} \|f\|_r.
\]
\end{proposition}
\begin{proof}
As mentioned at the beginning of this section, estimate \eqref{eq:errorFEVidar} holds for $k=1$ in the case of convex polygonal domains in dimensions two and three. The rest of the proof follows as for Proposition~\ref{thm:FE2}.
\end{proof}

We now prove our main results.

\begin{proof}[Proofs of Theorem \ref{thm:smooth1} and Theorem \ref{thm:convexPol1}]
Once the results on the quadrature rules and FE are established, the proof of our main result is straightforward and follows from the triangle inequality:
\begin{equation*}
\begin{split}
\|Q^s_r[w_{k,\tau}]-(-\Delta_\B)^{-s}f\|_{L^2(\Omega)}&=\|Q^s_r[w_{k,\tau}]-Q^s_r[e^{\cdot \Delta_\B }f]+Q^s_r[e^{\cdot \Delta_\B }f]-(-\Delta_\B)^{-s}f\|_{L^2(\Omega)}\\
&\leq \|Q^s_r[w_{k,\tau}]-Q^s_r[e^{\cdot \Delta_\B }f]\|_{L^2(\Omega)}+\|Q^s_r[e^{\cdot \Delta_\B }f]-(-\Delta_\B)^{-s}f\|_{L^2(\Omega)}.
\end{split}
\end{equation*}
The first term is bounded by the FE estimates given in Proposition \ref{thm:FE2} (resp. Proposition \ref{thm:FE3}), while the second term is bounded by the quadrature estimate of Theorem \ref{thm:interp}. 
\end{proof}

%%%%%%
\section{Nonhomogeneous boundary conditions}\label{sec:NHBC}
Following the approach introduced in~\cite{CuTeGGPa18}, we define now a family of nonlocal operators that can handle nonhomogeneous boundary conditions. More precisely, let $u:\Omega\to\R$ be such that $\B(u)=g$ for some given $g\in\partial \Omega$. Let $v:\Omega\times [0,\infty)$ denote the solution of the following heat equation with initial condition~$u$ and nonhomogeneous boundary conditions:
\begin{equation}
\label{eq:heatg}
\left\{
\begin{array}{lcl}
\partial_t v-\Delta v = 0 & \quad & (x,t)\in \Omega\times (0,\infty),\\
v(x,0)=u(x)& & x\in \Omega,\\
\B(v(\cdot,t))(x)=g & & (x,t) \in \partial \Omega\times [0,\infty).
\end{array}
\right.
\end{equation}
Generalizing \eqref{eq:FLdef2}, the nonlocal operator $\LBs$ on the given $u$ is defined as follows~\cite{CuTeGGPa18}:
\begin{equation*}\label{eq:fraclapnon}
\LBs [u](x):=\frac{1}{\Gamma(-s)}\int_{0}^\infty \left(v(x,t)-v(x,0)\right) \frac{d t}{t^{1+s}}, \quad \forall x\in \Omega.
\end{equation*}
Let $z$ be the steady-state of \eqref{eq:heatg}, that is, the solution of the local elliptic problem
\begin{equation*} \label{eq:ellip}
	\left\{
	\begin{array}{lcl}
	\Delta z = 0 & \quad & x\in \Omega\\
	\mathcal{B}(z)=g & & x \in \partial \Omega.
	\end{array}
	\right.
\end{equation*}
As proved in \cite{CuTeGGPa18}, if $z$ is regular enough we have that 
\begin{equation}\label{eq:relationHoNHo}
\LBs [u](x)=(-\Delta_{\B})^s [u-z](x) \qquad \textup{for all} \qquad x\in\Omega.
\end{equation}
We refer to \cite{GiTr01,lions_etal-book1968a} for results on the regularity of $z$ depending on the domain $\Omega$, the type of boundary condition $\B$ and the boundary data $g$. In Section 5.2 and 5.3 of~\cite{CuTeGGPa18}, some results and examples in this direction can also be found.

Associated to the operator $\LBs$ we consider the following nonhomogeneous elliptic problem:
\begin{equation}
	\label{eq:NHdirich}
	\left\{
	\begin{array}{lcl}
	\LBs[u](x) = f(x) & \quad & x\in \Omega\\
	\mathcal{B}(u)(x)=g(x) & & x \in \partial \Omega.
	\end{array}
	\right.
\end{equation}
In view of \eqref{eq:relationHoNHo}, problem \eqref{eq:NHdirich} can be reformulated in terms of $\psi:=u-z$ as  
\begin{equation}
	\label{eq:NHdirich2}
	\left\{
	\begin{array}{lcl}
	(-\Delta_{\B})^s [\psi](x) = f(x) & \quad & x\in \Omega\\
	\mathcal{B}(\psi)(x)=0 & & x \in \partial \Omega.
	\end{array}
	\right.
\end{equation}
The solution of \eqref{eq:NHdirich2} is $\psi=(-\Delta_\B)^{-s}f$ and hence the solution of \eqref{eq:NHdirich} can be expressed as
\begin{equation}\label{eq:reationBBproblem}
(\LBs)^{-1}[f](x):=(-\Delta_\B)^{-s}f(x)+z(x).
\end{equation}
Using \eqref{eq:reationBBproblem} and assuming that $z\in H^{l+1}(\Omega)$, if we let  $z_l$ denote a FE approximation of $z$ of order~$l+1$, then a natural numerical approach to the solution of \eqref{eq:NHdirich} is given by $Q^s_r[w_{k,\tau}]+z_l$ were $Q^s_r[w_{k,\tau}]$ is the numerical approximation of $(-\Delta_\B)^{-s}f$ studied in this manuscript. A simple triangle inequality allows us to use the results of our previous theorems to get error estimates of the form 
\[
\|Q^s_r[w_{k,\tau}]+z_l-(\LBs)^{-1}[f]\|_{L^2(\Omega)}\leq C ~\left(h^{\min\{k+1,a(\frac{r}{2}+s)\}}~ \|f\|_{\HH^r(\Omega)}+ h^{l+1}\|z\|_{H^{l+1}(\Omega)}\right).
\]

\section{Numerical tests}
\label{sec:numtests}
In this section, we provide some numerical evidence in support of our main theoretical results. We provide experiments both in one and two spatial dimensions, considering right-hand sides with different regularities and testing both the quadrature approximation and its combination with FE.

\textbf{One-dimensional test: $Q_r^s[e^{\cdot \Delta_\B}f]$ versus $Q_r^s[w_{k,\tau}]$.} We start by considering a one-dimensional example in which the analytic solution is known. Specifically, we consider as right-hand side for the fractional elliptic problem \eqref{EP}-\eqref{BC} on $\Omega=(0,1)$ with homogeneous Dirichlet boundary conditions the following function:
\begin{equation}\label{eq:firsteig}
f(x)= \pi^{2s} \sin(\pi x)=\lambda_1^s~ \varphi_1(x),
\end{equation}
with $\{\lambda_1, \varphi_1 \}$ the first eigenpair of the Dirichlet Laplacian. Then, the solution $u=\DBm f$ is trivially the first eigenfunction $\varphi_1$ itself. Note that, by definition, all eigenfunctions $\varphi_m$ belong to $ \HH^r$ for any $r\geq 0$ so that for this numerical test we are free to select the value of $r$ to be used in the quadrature approximation. 

In Figure~\ref{fig:1Deigf} we compare the error obtained when our quadrature is applied to the exact heat semigroup with the error obtained when our fully discrete method is used (that is, when the heat semigroup is in turn approximated via FE). In fact, when $f=\lambda_1^s\varphi_1$, the heat semigroup is nothing but $e^{t \Delta_\B}f=e^{-\lambda_1 t}\lambda_1^s\varphi_1$ and hence it can be evaluated exactly for all $t\geq 0$. For this particular example we set $r=1.5$, we consider three different values of $s$, and select $\varDelta t \sim h^2$ in both the semi-discrete and fully-discrete approximations. Note however, that in the semi-discrete case $h$ is nothing but an input parameter for the selection of the quadrature nodes location and no real spatial discretization is considered.

\begin{figure}[h!]
	\centering
	\includegraphics[width=.45\textwidth]{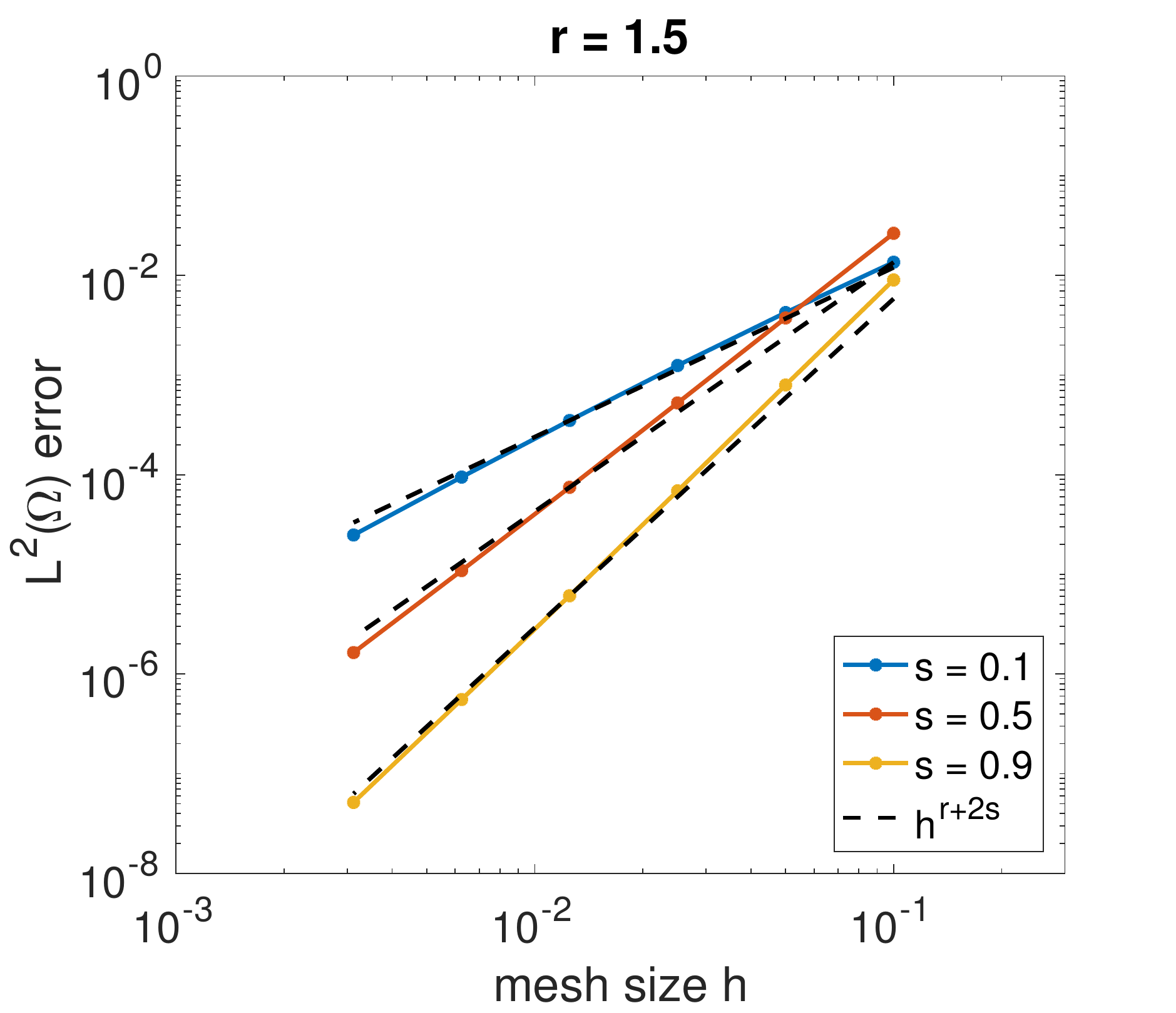}
	\includegraphics[width=.45\textwidth]{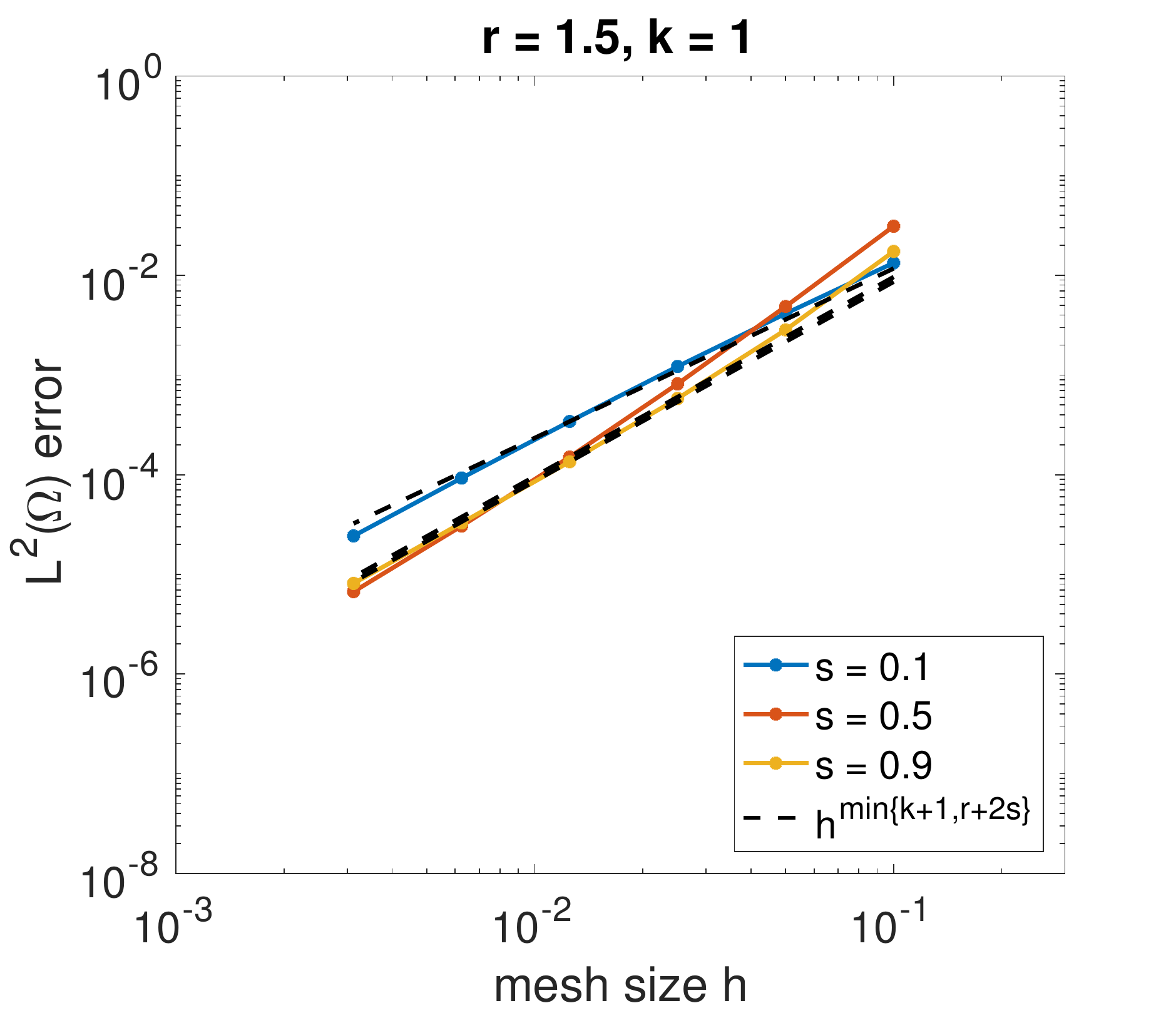}
	\caption{$\|(-\Delta_B)^{-s}f-Q_r^s[e^{\cdot \Delta_\B}f]\|_{0}$ (left) and $\|(-\Delta_B)^{-s}f-Q_r^s[w_{k,\tau}]\|_{0}$ (right) for different values of~$s$, $ \Omega=(0,1)$,  and $f$ given by \eqref{eq:firsteig}. The quadrature is computed with $r=1.5$ and $\varDelta t\sim h^2$, while the FE approximation $w_{k,\tau}$ is obtained with $k=1$ and $\tau=2$.}
	\label{fig:1Deigf}
\end{figure}

As we can see on the left of Figure~\ref{fig:1Deigf}, the error $\|\DBm f- Q_r^s[e^{\cdot \Delta_\B}f]\|_0$ as a function of the parameter~$h$ behaves like $\mathcal{O}(h^{r+2s})$ and hence it is perfectly in line with the result of Theorem~\ref{thm:interp} for $\varDelta t\sim h^a$ with $a=2$. On the other hand, when the FE approximation of the heat semigroup is introduced, as stated in Theorem~\ref{thm:smooth1} we expect the error to become bounded by the FE spatial discretization whenever $k+1<a(r/2+s)$. Since $r=1.5$, $a=2$, and $k=1$, this is observed  on the right of Figure~\ref{fig:1Deigf} for the cases $s=0.5$ and $s=0.9$, where the decay in $\|\DBm f- Q_r^s[w_{k,\tau}]\|_0$ behaves  as $\mathcal{O}(h^2)$. 

\textbf{One-dimensional test: quadratic convergence and computational cost.} In Figure~\ref{fig:1Deigf2}, 
we let~$f$ be as in~\eqref{eq:firsteig} but we choose a higher value for the regularity parameter (namely, $r=2$).  Errors  are always computed with respect to the fully-discrete approximation $Q_r^s[w_{k,\tau}]$. 
The parameter choice for  Figure~\ref{fig:1Deigf2} (left) is $r=2$, $k=1$, $a=2$ and results in an error that is  controlled by the FE approximation in space for all values of $s$. Therefore, as discussed in Section~\ref{sec:computcost}, the parameter $a$ could be tuned to reduce computational costs while still preserving quadratic convergence. The ``optimal'' choice of $a$ in this case is $a_{\rm opt}=2/(1+s)$. The corresponding numerical error is shown in Figure~\ref{fig:1Deigf2} (right), while the number of quadrature points required by our numerical method is given in Table~\ref{tab:cost} exhibiting significant advantages for $a_{\rm opt}$.

\begin{figure}[h!]
	\centering
	\includegraphics[width=.45\textwidth]{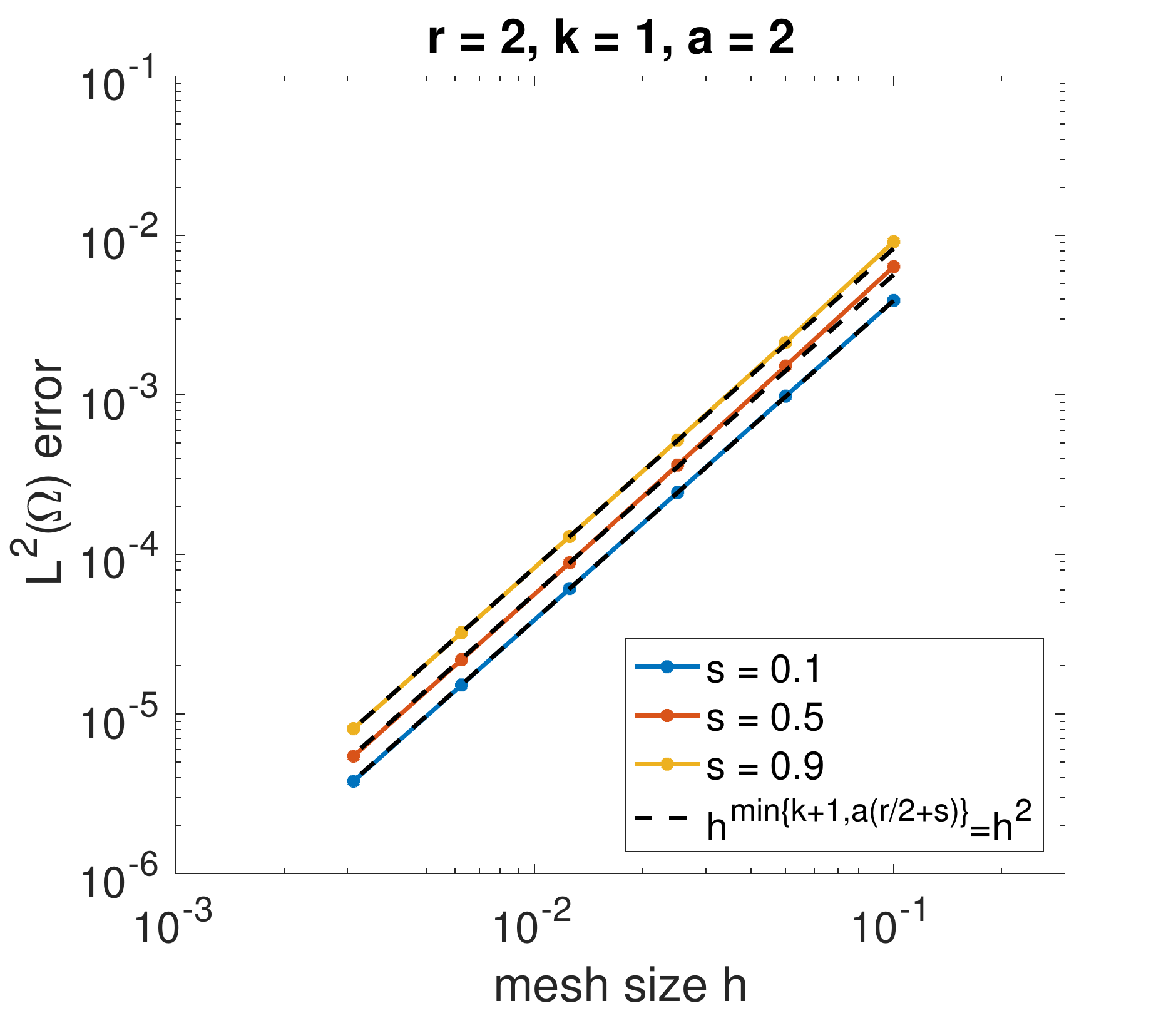}
	\includegraphics[width=.45\textwidth]{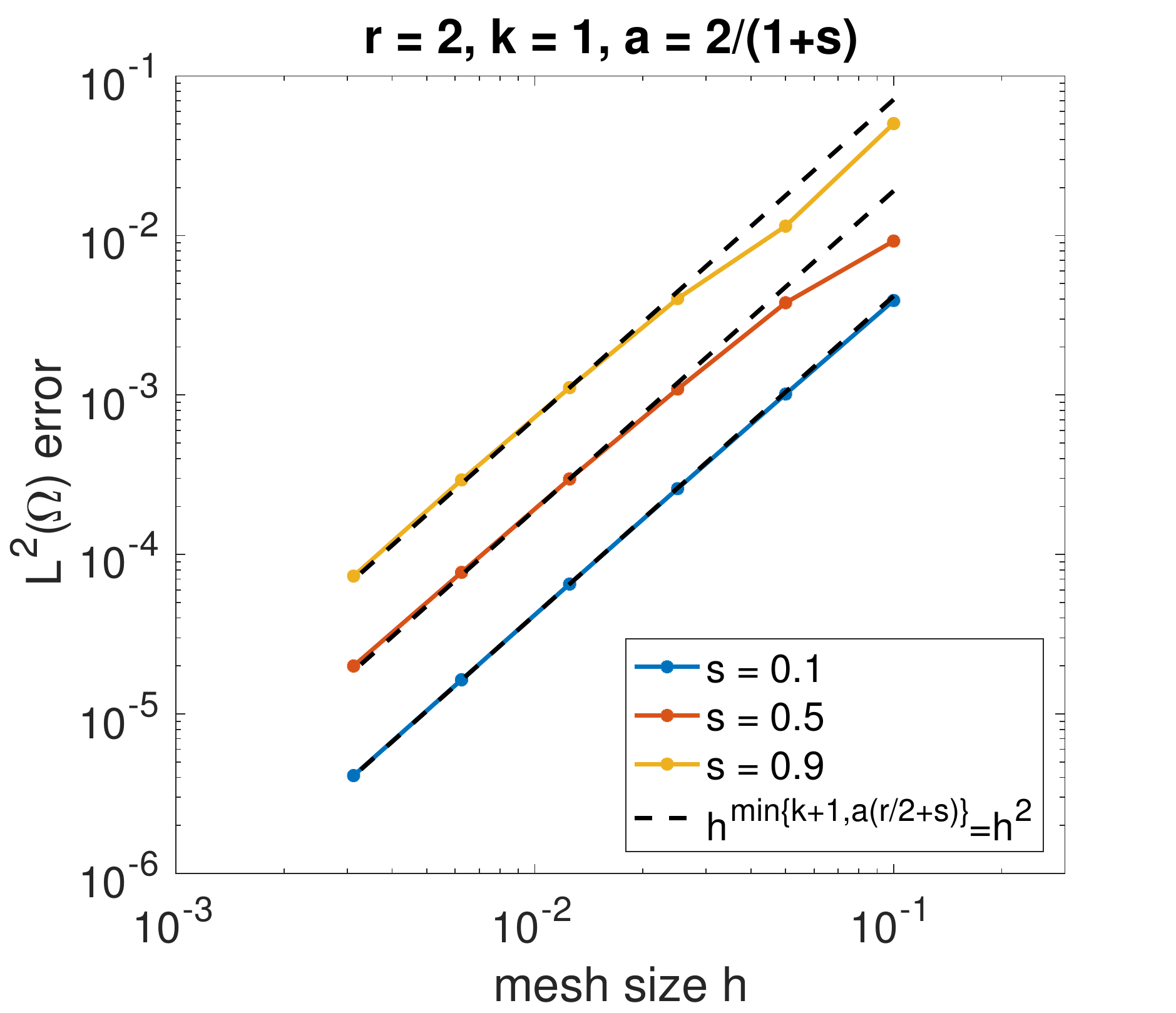}
	\caption{$\|(-\Delta_B)^{-s}f-Q_r^s[w_{k,\tau}]\|_0$ for different values of $s$ and $f$ defined by~\eqref{eq:firsteig} on $\Omega=(0,1)$. We set $r=2$, $k=1$, $\tau=2$,  $a=2$ (left), and $a=\frac{2}{(1+s)}$ (right).}
	\label{fig:1Deigf2}
\end{figure}
\begin{table}[h!]
	\centering
	\begin{tabular}{|c|cc|cc|cc|}\hline
		& \multicolumn{2}{c|}{$s=0.1$}&\multicolumn{2}{c|}{$s=0.5$}&\multicolumn{2}{c|}{$s=0.9$}\\ 
		$h$&$a=2$&$a_{\rm opt}$&$a=2$&$a_{\rm opt}$&$a=2$&$a_{\rm opt}$\\ \hline
		0.1&38&23&51&7&66&3\\
		0.05&200&105&279&23&369&10\\
		0.025&1012&465&1432&75&1899&26\\
		0.0125&4921&1991&7011&233&9290&67\\
		0.00625&23248&8293&33217&702&43945&167\\ 
		0.003125&107462&33809&153684&2058&202959&407\\ \hline
	\end{tabular}
	\caption{Number of quadrature nodes $N_T$ as a function of $s$ and $h$ for two choices of $a$ giving quadratic convergence: $a=2$ and $a_{\rm opt}=\frac{2}{(1+s)}$. The other simulation parameters are $k=1$, $\tau=2$, and $r=2$.}
	\label{tab:cost}
\end{table}

\textbf{One and two-dimensional tests: low regularity (and incompatible) right-hand side.} We then consider two examples of incompatible, low regularity data $f$ in one and two-dimensional domains, respectively. The first case, illustrated in Figure~\ref{fig:low1D}, is of $f(x)=x^{-0.5+\epsilon}$ on $\Omega=(0,1)$, which is merely a function in $L^2(\Omega)$. Here, the parameter $r$ cannot be chosen freely but rather corresponds to the regularity of the considered right-hand side, that is, $r=0$. The exact solution can be expressed as a series, namely, $u=\sum_{m=1}^\infty \lambda_m^{-s}\hat{f}_m \varphi_m$, involving the eigenpairs of the Dirichlet Laplacian on the interval considered (explicitly known on one-dimensional intervals). However, the value of the integral coefficients $\hat{f}_m$ cannot be computed exactly and a numerical approximation would have to be introduced. Nevertheless, by computing a reference solution on a very fine mesh (i.e., $h=7.8125\cdot 10^{-5}$) and calculating the approximation error of our method (with $k=1$, $\tau=1$, and $a=1$) against it, we still recover the expected decay rates (i.e., $\mathcal{O}(h^s)$).

\begin{figure}[h!]
	\centering
	\includegraphics[width=.45\textwidth]{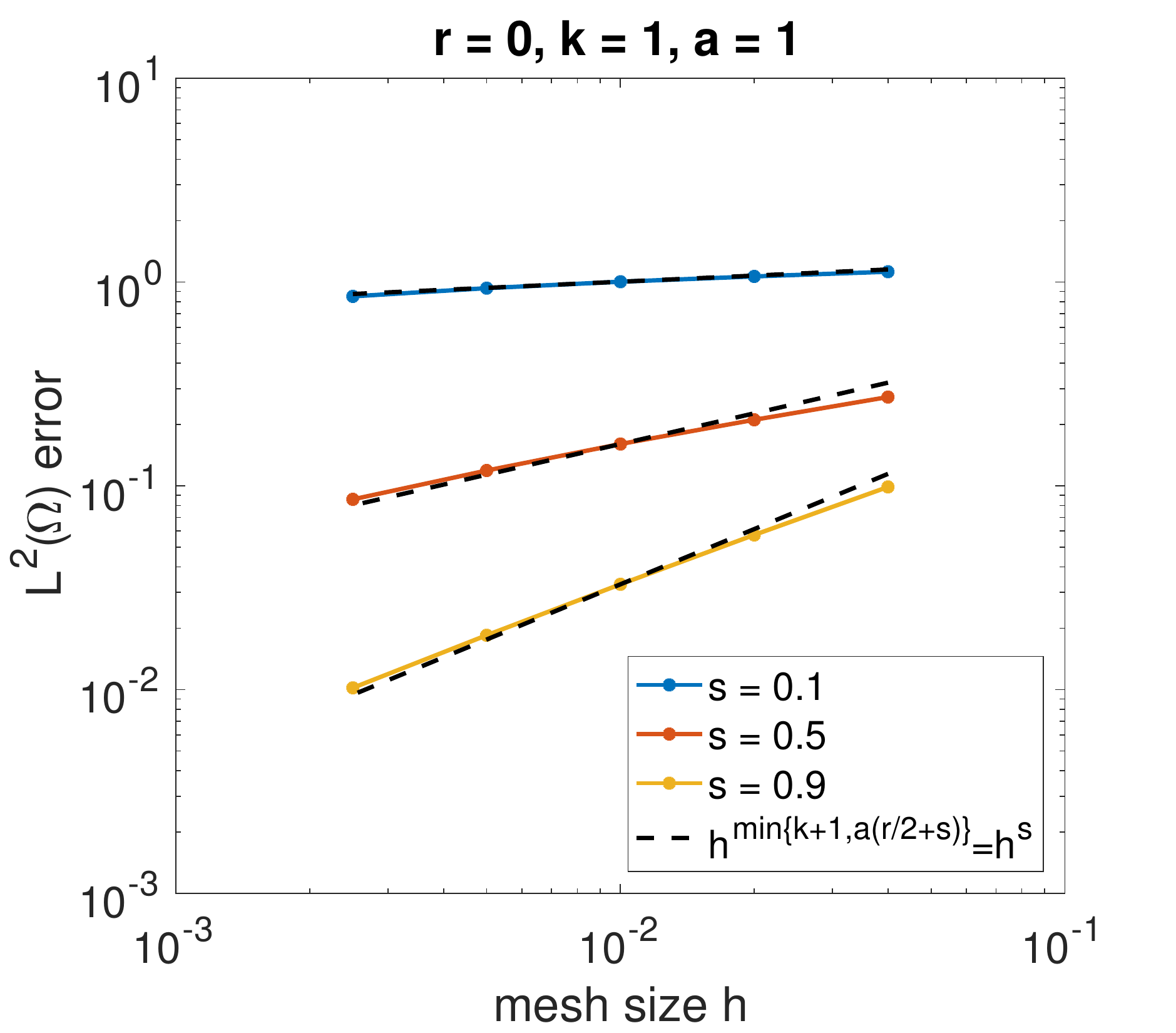}
	\caption{Error decay in the approximation of $(-\Delta_B)^{-s}f$ for different values of $s$, with $\Omega=(0,1)$ and $f=x^{-0.5+\epsilon} \in L^2(\Omega)$ so that $r=0$. Numerical approximation computed with $k=1$, $\tau=1$, and $a=1$.}
	\label{fig:low1D}
\end{figure}

The second example of right-hand side of \eqref{EP}-\eqref{BC} with low regularity is given by the ``checkerboard'' function
\[
f(x,y)=\left\{
\begin{array}{rl}
1 & \text{if } (x-0.5)(y-0.5)>0\\
-1 & \text{elsewhere}, 
\end{array}
\right.
\]
on the two-dimensional domain $\Omega=(0,1)\times(0,1)$. In this case, $f\in \HH^r$ with $r=0.5-\epsilon$ for all $\epsilon>0$ and once again the exact solution can be expressed as a series involving the eigenpairs of the Dirichlet Laplacian (still known explicitly in the case of square $\Omega$). Moreover, this time the integral coefficients $\hat{u}_m$ can be computed exactly and the exact solution can be easily obtained by truncating its series expansion after a sufficiently large number of terms. We do so by using $102400$ terms and plot the numerical error obtained through our approximation strategy with two values of $k$ in Figure~\ref{fig:low2D}. We refer the reader to the work by Bonito and Pasciak~\cite{Bon2015} for an alternative discretization method used to solve this two-dimensional problem and for illustrations of the solution profile for some selected values of $s$.

\begin{figure}[h!]
	\centering
	\includegraphics[width=.45\textwidth]{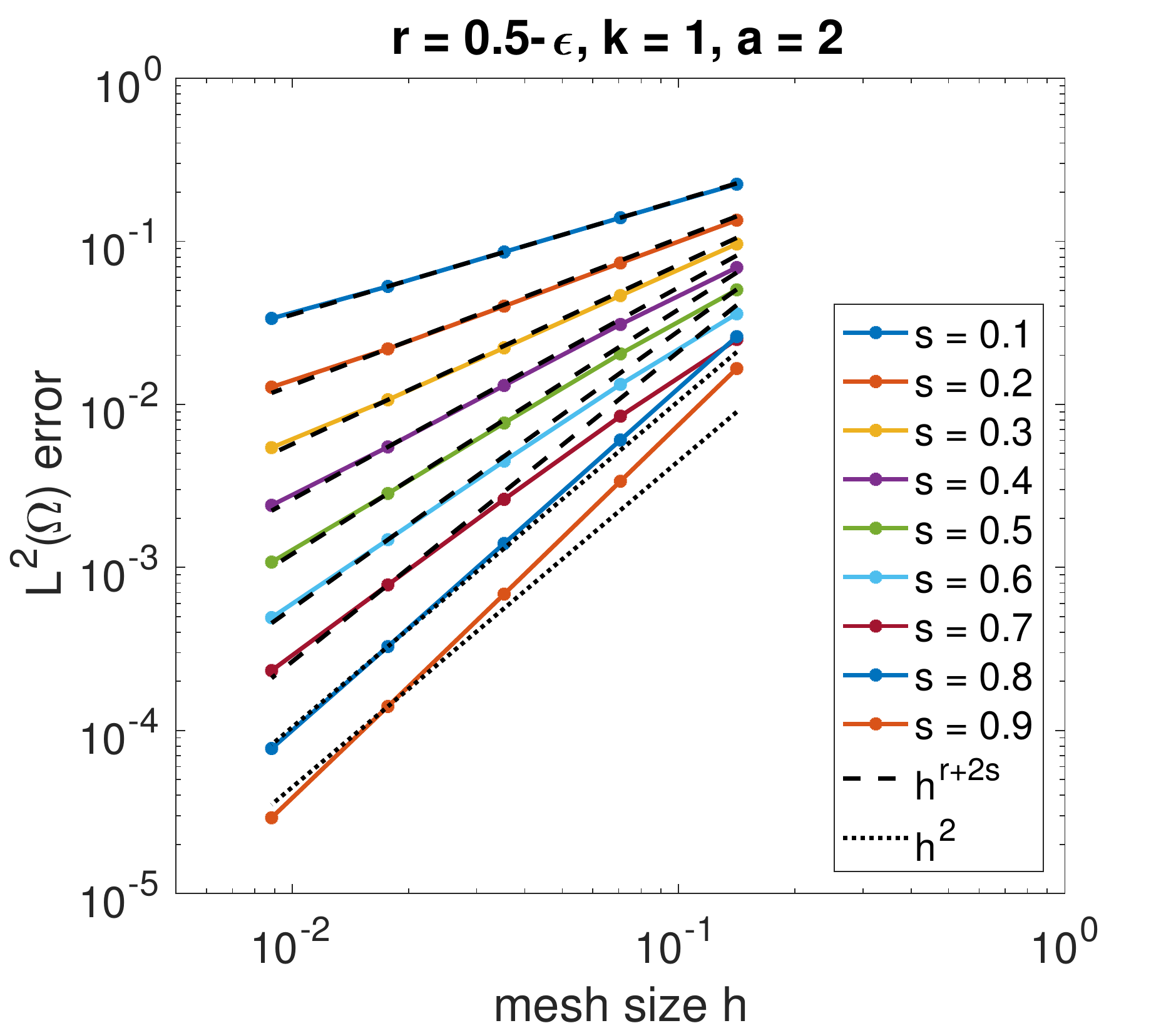}
	\includegraphics[width=.45\textwidth]{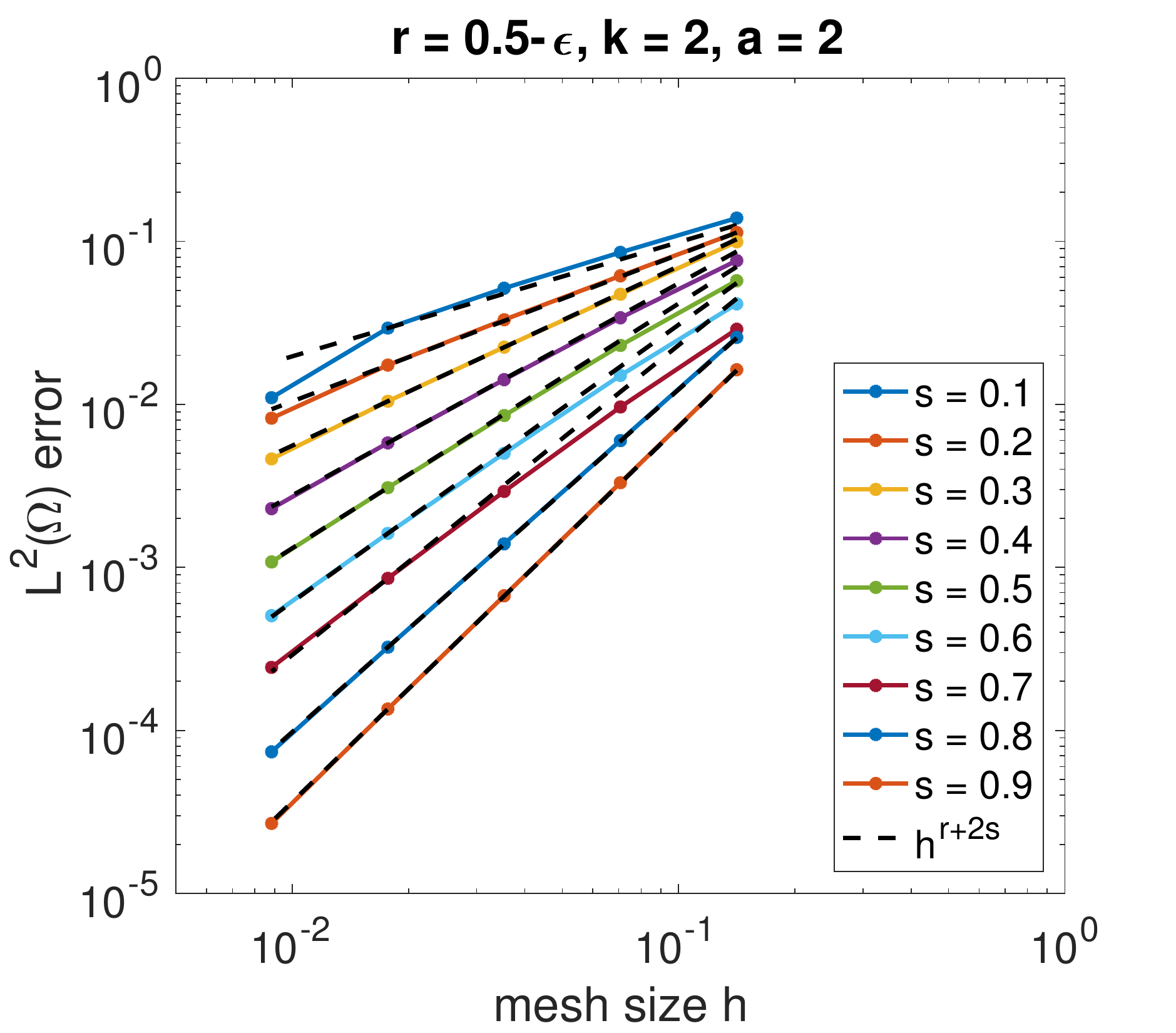}
	\caption{Error decay in the approximation of $(-\Delta_B)^{-s}f$ for different values of $s$ and with $f$ the ``checkerboard'' function on $\Omega=(0,1)\times(0,1)$ so that $r=0.5-\epsilon$. Numerical approximation computed with either $k=1$ (left) or $k=2$ (right), $\tau=2$, and $a=2$.}
	\label{fig:low2D}
\end{figure}

On the left of Figure~\ref{fig:low2D}, we show the error obtained for nine different values of $s$ when linear FE are considered and obtain results in agreement with the ones proposed by Bonito and Pasciak~\cite{Bon2015}. In fact, when $s<(2-r)/2=0.75$, we observe that the error decay of our numerical method is essentially controlled by the quadrature approximation, decaying as $\mathcal{O}(h^{r+2s})$ (since $a=2$). On the other hand, when $s>(2-r)/2$ (i.e., for $s=0.8$ and $s=0.9$), we expect the bound coming from the FE (namely $\mathcal{O}(h^{2})$) to become the dominant term and hence asymptotically control the decay. Our numerical tests for these two values of $s$ seem actually to behave a little better than expected but this is likely due to the fact that the asymptotic regime was not yet reached by the smallest value of $h$ considered. 
Nevertheless, as shown on the right of Figure~\ref{fig:low2D}, this limitation is overcome when quadratic FE are considered, since the error decays as $\mathcal{O}(h^{r+2s})$ for all considered values of $s$. Therefore, in this particular example, even though the considered domain is non-smooth, the choice of second order FE in space allows us to obtain an improvement in the result of Theorem~\ref{thm:convexPol1} and hence produce improved convergence rates for values of $s>0.75$. 

Although a detailed discussion of this is out of the scope of the presented work, we believe that this improved result is due to the fact that well-known limited regularity estimates for the standard Dirichlet elliptic problem can be improved on plane convex polygonal domains for which a suitable control on interior angles holds (e.g., see \cite{bacuta_etal-2003}). This feature is likely to entail better FE estimates in space for the corresponding heat equation solution that in turn would carry over to the approximation of the fractional Poisson solution proposed in this work (as shown in Section~\ref{sec:FE}). Specifically, through the use of FE of order $k>1$, we expect to be able to obtain, on plane convex polygonal domains, errors $\mathcal{O}(h^{2+\gamma})$ for some $0<\gamma<1$ depending on the improvement on the elliptic regularity allowed by the angle control, instead of simply $\mathcal{O}(h^{2})$, so that the error decay of our method would become of the form $\mathcal{O}(h^{\min\{k+1,2+\gamma,a(r/2+s)\}})$.

\textbf{One-dimensional test: different types of boundary conditions.} To conclude this section we show an example of how the solution to the fractional Poisson problem~\eqref{EP}-\eqref{BC} varies when the fractional operator is coupled to different boundary conditions. The domain $\Omega=(0,1)$, the fractional power $s=0.3$, and the right-hand side $f$ of the problem are the same for all three fractional operators. The datum is chosen as 
\[
f(x)=\left\{
\begin{array}{rl}
1 & \text{if } (x-0.5)<0\\
-1 & \text{elsewhere}, 
\end{array}
\right.
\] 
and hence (as in the two-dimensional case previously considered) $r=0.5-\epsilon$. The datum $f$ and the corresponding solution to the fractional Poisson equation with either homogeneous Dirichlet, Neumann, or Robin (with $\kappa=1$ on $\partial \Omega$) boundary conditions are shown on the left of Figure~\ref{fig:BCs}. The substantial difference between the three cases considered is clearly visible. The corresponding approximation error as a function of the mesh size $h$ is reported in the plot on the right of Figure~\ref{fig:BCs} and once again is in perfect agreement with the result of our main theorem.

\begin{figure}[h!]
	\centering
	\includegraphics[width=.45\textwidth]{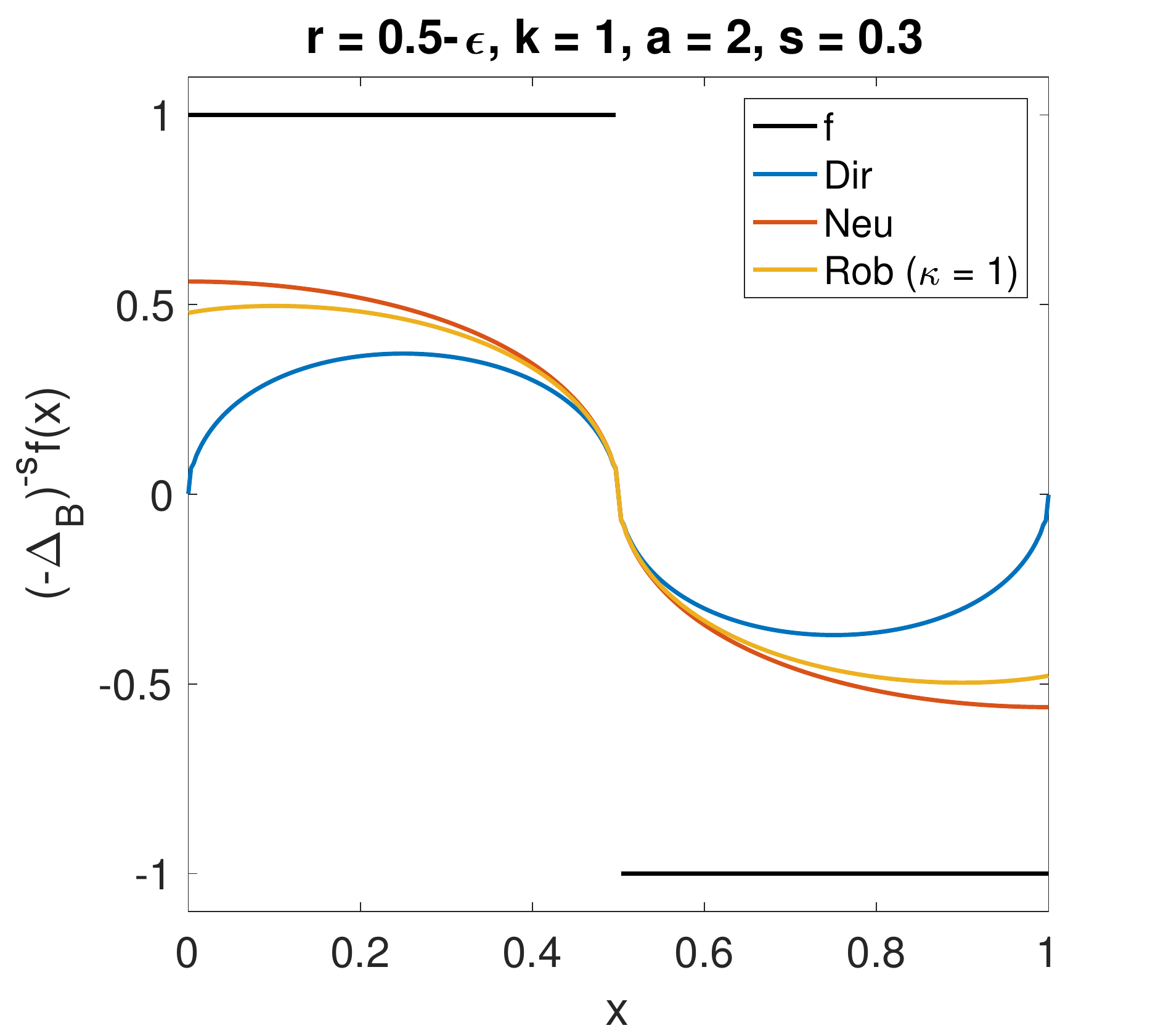}
	\includegraphics[width=.45\textwidth]{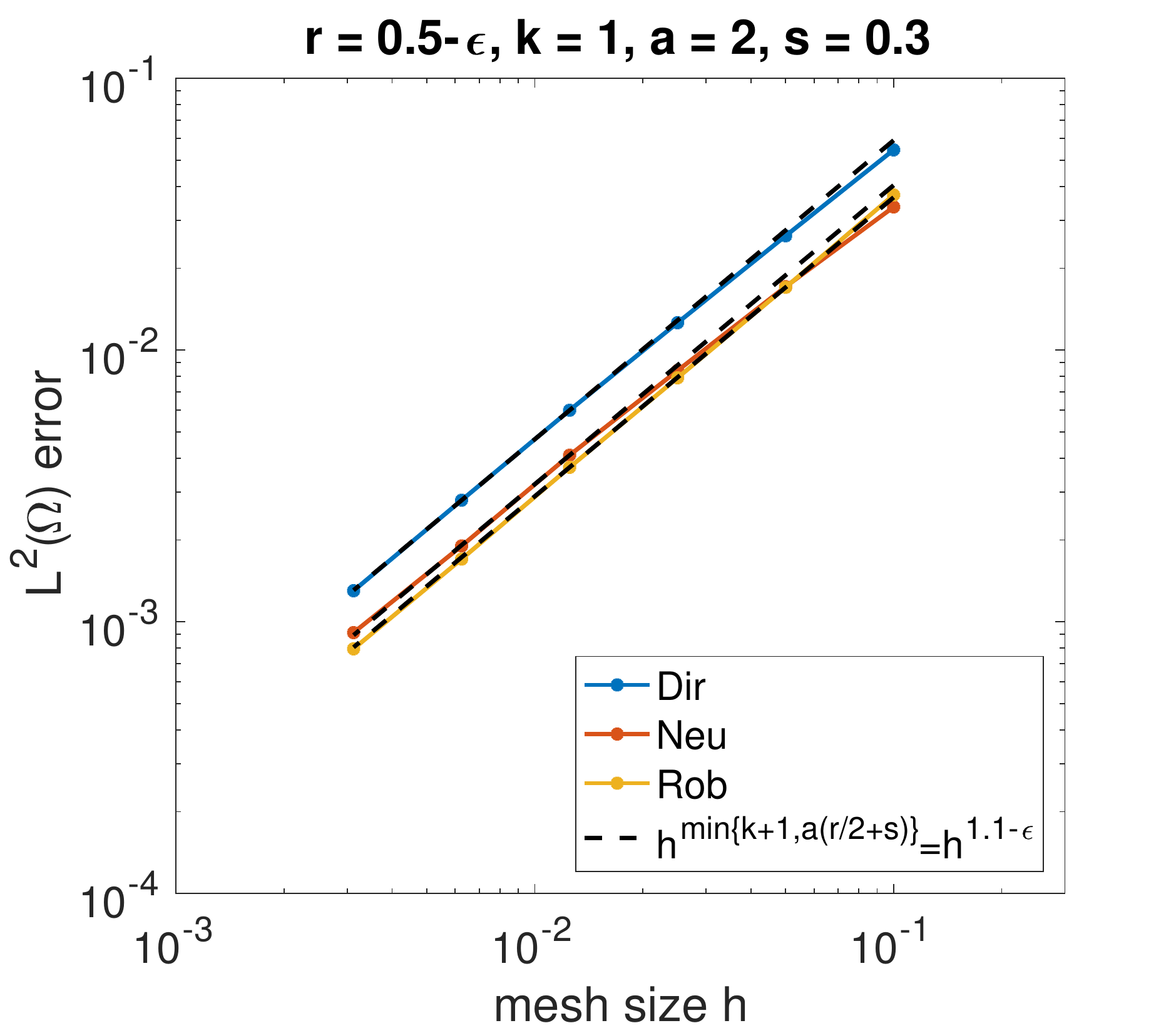}
	\caption{Left: Solution $\DBm f$ of the fractional Poisson equation on $\Omega=(0,1)$ with the same data $f$ and same fractional exponent $s$. Right: Error decay in the approximation of $(-\Delta_B)^{-s}f$ corresponding to the three boundary operators (Dirichlet, Neumann, and Robin with $\kappa=1$ on $\partial \Omega$). Numerical approximation computed in all cases with $r=0.5-\epsilon$, $k=1$, $\tau=2$, and $a=2$.}
	\label{fig:BCs}
\end{figure}

\section{Comments, possible extensions and open problems}
\label{sec:comments}
We conclude this manuscript by providing some possible directions in which the current work could be further developed. 

\noindent$\bullet$
As mentioned in the introduction, the method of semigroups is not the only option available in the literature to solve numerically the fractional Poisson problem~(\ref{EP})-(\ref{BC}). While here our aim is to address what we see as a gap in the literature by exploiting the well-established method of semigroups in order to propose a new numerical approach for the solution of the considered fractional problem, questions regarding the optimal implementation of the method and its performance compared to other strategies remain open. 
For example, in the present work we consider a uniform discretization for the integration variable $t$ because of its simplicity. This already allows us to obtain optimal or near optimal error estimates, but clearly more sophisticated discretization approaches could be used (such as adaptive grids, better exploiting the heat kernel decay in $t$ and the singularity of the integration measure) to reduce the computational cost of the proposed algorithm. 

Overall, the main differences we see in adopting the method of semigroups compared to other strategies proposed in the literature could be summarised as follows. Our method does not rely on a direct computation of the eigenpairs (as in \cite{song_etal-2017}) of $-\Delta_\B$, that can be a heavy computational task, especially for irregular geometries.
Compared to the extension method in \cite{NoOtSa15}, our approach can handle all powers $s\in(-1,1)$ using the local semigroup $e^{t\Delta}$ which does not depend on $s$, while in \cite{NoOtSa15} the associated local problem changes with $s$. Moreover, while our approach requires the computation of the heat equation solution at various time points (either by solving a sequence of linear system in the implicit case or simply by computing matrix-vector products in the explicit one), the authors of~\cite{NoOtSa15} need to solve only one local problem that involves however an additional spatial dimension, entailing the solution of a much larger linear system of equations. Finally, both our method and the one relying on \eqref{eq:invPowerBo} by Bonito and Pasciak~\cite{Bon2015} consist in approximating the solution of~(\ref{EP})-(\ref{BC}) as a weighted sum of vectors. In our approach, these vectors can be computed iteratively, either by solving linear systems of equations associated to the same matrix (implicit case) or through the computation of matrix-vector products with a fixed matrix (explicit case). On the contrary, the method of \cite{Bon2015} requires the solution of a sequence of linear systems all associated to different matrices, but is well-suited to parallel implementation (clearly an advantage in presence of multi-core processors or HPC resources).
 
A detailed comparison of all existing methods for the same operator would certainly be of great practical interest. However, we believe that, for such a numerical comparison to be effective, extensive knowledge of each and every approach would be required. Moreover, this analysis would be especially meaningful if all algorithms were implemented by the respective developers (as to take advantage of their best features). This should be a joint effort of the numerical community dealing with this topic that could surely be beneficial to many researchers interested in the applications of these methods.  

\noindent$\bullet$ An interesting question is whether this numerical approach can handle right-hand sides  $f\in \HH^r(\Omega)$ with $r\in(-2s,0)$ since the analytical solution $u := (-\Delta_\B)^{-s} f$ of \eqref{EP}-\eqref{BC} exists and belongs to $\HH^{r+2s}(\Omega)$, a Sobolev space with positive exponent. We can extend the results of Lemma \ref{lem:splitinterp} to this range too. More precisely, for $r\in(-2s,0)$, one can define the interpolant 
\[
\mI_r[\rho](t)= \left\{ \begin{array}{ll}
0 &t\in[t_0, t_1),\\
\rho(t_j)& t\in[t_j,t_{j+1}), \forall j \geq 1.\\
\end{array} \right.
\]
and the associated quadrature $Q_r^s$. With these considerations, the results of Theorem \ref{thm:interp} also hold in the range $r\in(-2s,0]$. However, to the best of our knowledge, FE estimates in the form of \eqref{eq:errorFEVidar} are not available in the literature for initial data in negative Sobolev spaces. Hence, a generalization of our result to the considered range would perhaps require an \textit{ad-hoc} modification of our strategy to accommodate also the scenario of right-hand sides in negative fractional Sobolev spaces.

\noindent$\bullet$ Another possible extension of the proposed results concerns non-convex  polygonal  domains. In this context, the limited regularity of the standard elliptic problem due to the presence of singularities near the re-entrant corners of the domain typically results in an additional reduction (with respect to the case of convex domains) in the optimal order of convergence of the FE approximation. In fact, on quasi-uniform grids for non-convex domains, \eqref{eq:errorFEVidar} no longer holds even for linear FE ($k=1$). Nevertheless, as shown by Chatzipantelidis et al.~in~\cite{chatzipantelidis_etal-bit-2006}, by suitably refining the discretization mesh near the re-entrant corners, optimal order of convergence $\mathcal{O}(h^2)$ can be recovered. Therefore, we expect that by suitably adapting the FE settings (and ensuring a sufficiently small choice for $\varDelta t$ in our quadrature rule), our main result would also hold with $k=1$ on non-convex polygonal domains. 

\noindent$\bullet$ In view of the results and strategies of the present article, it seems very natural and straightforward to extend our results to powers of more general second order differential operators and boundary conditions. The results on quadrature rules do not need any modification (beside a natural change in the functional space settings). FE estimates like \eqref{eq:errorFEVidar}, which are one of the main tools in our proofs, hold for more general operators as shown in \cite{Vid97} and \cite{QV2008}.

\noindent$\bullet$ As previously mentioned, the method by Bonito and Pasciak \cite{Bon2015} can be implemented in parallel, thus allowing faster computations. One possible extension of this work would be the investigation of parallel computing methods for the strategy presented in this work and the investigation of other possible approaches to reduce computational costs of our technique.

\section*{Acknowledgements}
We would like to thank the anonymous reviewers of this work for their helpful comments to improve the quality of our manuscript. 
This research is supported by the Basque Government through the BERC 2018-2021 program and by the Spanish Ministry of Science, Innovation and Universities: BCAM Severo Ochoa accreditation SEV-2017-0718. N.C. and L.G.G. are also supported by the Spanish ``Plan Estatal de Investigaci\'on, Desarrollo e Innovaci\'on Orientada a los Retos de la Sociedad'' under Grant BELEMET - Brain ELEctro-METabolic modeling and numerical approximation (MTM2015-69992-R) and by the Spanish Ministry of Economics and Competitiveness MINECO through the grant RTI2018-093416-B-I00, while F.d.T. is also supported by the Toppforsk (research excellence) project Waves and Nonlinear Phenomena (WaNP), grant no. 250070 from the Research Council of Norway, by the MEC-Juan de la Cierva postdoctoral fellowship number FJCI-2016-30148 and by the Spanish research project PGC2018-094522-N-100 from the MICINNU.

\setlength{\itemsep}{-1em}
\setlength{\parskip}{0em}
\footnotesize
\bibliographystyle{plain}

\Addresses

\end{document}